\documentclass[11pt]{amsart}

\usepackage{etex}
\usepackage{amsmath, amssymb}
\usepackage{array}
\usepackage[frame,cmtip,arrow,matrix,line,graph,curve]{xy}
\usepackage{graphpap, color, paralist, pstricks}
\usepackage[mathscr]{eucal}
\usepackage[pdftex]{graphicx}
\usepackage[pdftex,colorlinks,backref=page,citecolor=blue]{hyperref}
\usepackage{tikz}

\newtheorem{theorem}{Theorem}[section]
\newtheorem{proposition}[theorem]{Proposition}
\newtheorem{corollary}[theorem]{Corollary}
\newtheorem{lemma}[theorem]{Lemma}

\newtheorem{conjecture}[theorem]{Conjecture}

\theoremstyle{definition}
\newtheorem{definition}[theorem]{Definition}
\newtheorem{example}[theorem]{Example}

\newtheorem{question}[theorem]{Question}
\newtheorem{remark}[theorem]{Remark}

\newtheorem{algorithm}[theorem]{Algorithm}

\newcommand{\PP}{\mathbb{P}}
\newcommand{\QQ}{\mathbb{Q}}
\newcommand{\CC}{\mathbb{C}}

\newcommand{\ZZ}{\mathbb{Z}}
\newcommand{\NN}{\mathbb{N}}

\newcommand{\cO}{\mathcal{O} }

\newcommand{\cC}{\mathcal{C} }
\newcommand{\cE}{\mathcal{E} }
\newcommand{\cF}{\mathcal{F} }

\newcommand{\cI}{\mathcal{I} }
\newcommand{\cM}{\mathcal{M} }

\newcommand{\cP}{\mathcal{P} }

\newcommand{\cS}{\mathcal{S} }

\newcommand{\rH}{\mathrm{H} }

\newcommand{\proj}{\mathrm{Proj}\;}

\newcommand{\Mznb}{\overline{\mathrm{M}}_{0,n}}
\newcommand{\Mznt}{\widetilde{\mathrm{M}}_{0,n}}
\newcommand{\Mzn}{\mathrm{M}_{0,n}}
\newcommand{\Mza}{\overline{\mathrm{M}}_{0,A}}

\newcommand*\circled[1]{\tikz[baseline=(char.base)]{
    \node[shape=circle,draw,inner sep=1pt] (char) {#1};}}

\def\SL{\mathrm{SL}}

\def\git{/\!/ }

\makeatletter
\providecommand{\leftsquigarrow}{%
  \mathrel{\mathpalette\reflect@squig\relax}%
}
\newcommand{\reflect@squig}[2]{%
  \reflectbox{$\m@th#1\rightsquigarrow$}%
}
\makeatother

\hypersetup{%
pdftitle={Birational contractions of $\Mznb$ and combinatorics of extremal assignments},%
pdfauthor={Han-Bom Moon, Charles Summers, James von Albade, Ranze Xie},%
pdfkeywords={moduli space, rational curves, birational geometry, Mori's program, extremal assignment, intersecting family, complete multipartite graph, integer partition},%
citecolor=blue,%
linkcolor=blue,%
}

\begin{document}

\title{Birational contractions of $\Mznb$ and combinatorics of extremal assignments}
\date{\today}

\author{Han-Bom Moon}
\address{Department of Mathematics, Fordham University, Bronx, NY 10458}
\email{hmoon8@fordham.edu}

\author{Charles Summers}
\address{Department of Mathematics, Fordham University, Bronx, NY 10458}
\email{csummers1@fordham.edu}

\author{James von Albade}
\address{Department of Mathematics, Fordham University, Bronx, NY 10458}
\email{jvonalbade@fordham.edu}

\author{Ranze Xie}
\address{Department of Mathematics, Fordham University, Bronx, NY 10458}
\email{jxie17@fordham.edu}

\begin{abstract}
From Smyth's classification, modular compactifications of pointed smooth rational curves are indexed by combinatorial data, so-called extremal assignments. We explore their combinatorial structures and show that any extremal assignment is a finite union of atomic extremal assignments. We discuss a connection with the birational geometry of the moduli space of stable pointed curves. As applications, we study three special classes of extremal assignments: smooth, toric, and invariant with respect to the symmetric group action. We identify them with three combinatorial objects: simple intersecting families, complete multipartite graphs, and special families of integer partitions, respectively. 
\end{abstract}

\maketitle


\section{Introduction}

A fascinating fact about the moduli space $\Mznb$ of stable $n$-pointed rational curves is that it has rich combinatorial structures. There is a natural stratification indexed by the set of stable $n$-labeled trees, or matroid polytope decompositions (\cite{Kap93b}). The closure of each stratum is isomorphic to a product of $\Mznb$'s with a small number of marked points, and the universal family over $\Mznb$ is isomorphic to $\overline{\mathrm{M}}_{0, n+1}$, so there are two inductive structures. The limit computation on $\Mznb$ can be explained in terms of Bruhat-Tits building (\cite{Kap93a}). A natural connection with root systems of type A was also observed (\cite{Sek96}).

Recently Smyth found yet another example of an interplay between geometry on $\Mznb$ and combinatorics. In \cite{Smy13}, he gave a classification of modular compactifications of the moduli space $\mathrm{M}_{0, n}$ of smooth rational curves with distinct $n$-points in the algebraic stack of all pointed curves. He showed that these compactifications can be indexed by combinatorial data, so-called \emph{extremal assignments}. An extremal assignment $Z$ is a collection of subsets $Z(G) \subset V(G)$ for every stable $n$-labeled tree $G$, with two conditions (Definition \ref{def:assignment}):
\begin{enumerate}
\item $Z(G) \ne V(G)$;
\item For any contraction $G \rightsquigarrow G'$ which contracts $\{v_{1}, v_{2}, \cdots, v_{k}\} \subset V(G)$ to $v' \in V(G')$, $v_{1}, v_{2}, \cdots, v_{k} \in Z(G)$ if and only if $v' \in G'$. 
\end{enumerate}
For each extremal assignment $Z$, Smyth constructed a moduli space $\Mznb(Z)$ of $n$-pointed rational curves with certain singularities, in the category of algebraic spaces. It is shown that there is a dominant regular map $\Mznb \to \Mznb(Z)$ which preserves $\Mzn$ (Proposition \ref{prop:functoriality}). Thus $\Mznb(Z)$ is a \emph{birational contraction} of $\Mznb$.

The aim of this paper is to study the combinatorics of extremal assignments and to translate the result in terms of birational contractions of $\Mznb$. 

\subsection{Structure theorem of extremal assignments}\label{ssec:stateofresult}

Although the definition of an extremal assignment is simple and natural, answering the following natural questions is surprisingly subtle. 

\begin{question}\label{que:fundamentalquestion}
Let $G_{1}, G_{2}, \cdots, G_{k}$ be stable $n$-labeled trees and let $v_{1} \in V(G_{1}), v_{2} \in V(G_{2}), \cdots, v_{k} \in V(G_{k})$. 
\begin{enumerate}
\item Is there an extremal assignment $Z$ such that $v_{i} \in Z(G_{i})$?
\item Describe the smallest $Z$ such that $v_{i} \in Z(G_{i})$. Or equivalently, for any $G$ and any $v \in V(G)$, determine whether $v$ must be assigned or not. 
\end{enumerate}
\end{question}

To answer these questions, first we prove a structure theorem of extremal assignments. An \emph{atomic} extremal assignment is the smallest extremal assignment with an assigned vertex. We will explain relevant notations (mainly for the partial ordering of set partitions) in the statement below in Section \ref{sec:structuretheorem}. 

\begin{theorem}\label{thm:structurethmintro}
\begin{enumerate}
\item There is a one-to-one correspondence between the set of atomic extremal assignments and the set of set partitions $P$ of $[n]$ with $3 \le |P| \le n-1$. 
\item Any extremal assignment $Z$ of order $n$ is a union of finitely many atomic extremal assignments. 
\item For atomic extremal assignments $Z_{1}, Z_{2}, \cdots, Z_{k}$ and the corresponding set partitions $P_{1}, P_{2}, \cdots, P_{k}$, the union $Z := \bigcup_{i=1}^{k}Z_{i}$ is an extremal assignment if and only if for any two $Q_{1} \preceq P_{i}$ and $Q_{2} \preceq P_{j}$ with the tight upper bound $R$, there is $P_{\ell}$ such that $R \le P_{\ell}$.
\end{enumerate}
\end{theorem}

Based on Theorem \ref{thm:structurethmintro}, we provide an algorithm and its implementation for finding the smallest $Z$ in Question \ref{que:fundamentalquestion}. 

\subsection{Smooth models and simple intersecting families}

As applications of the structure theorem, we explore three special cases of extremal assignments and their associated birational models. 

The first natural class is the collection of smooth models. We give a characterization of extremal assignments that provide smooth models. A \emph{smooth extremal assignment} is an extremal assignment $Z$ such that for any $G$ and $v \in Z(G)$, there is $G'$ with two vertices and $v' \in Z(G')$ such that $G \rightsquigarrow G'$ and $v \rightsquigarrow v'$. 

\begin{theorem}\label{thm:divisorialintro}
\begin{enumerate}
\item An extremal assignment $Z$ provides a smooth birational contraction $\Mznb(Z)$ if and only if $Z$ is equivalent to a smooth extremal assignment. 
\item There is a one-to-one correspondence between the set of smooth extremal assignments of order $n$ and the set of simple intersecting families of order $n$, rank $\le n-2$, and antirank $\ge 2$. 
\end{enumerate}
\end{theorem}

A simple intersecting family is a special class of hypergraphs (\cite[Section 1.3]{Ber89}) such that any two hyperedges properly intersect. 

\subsection{Toric models and complete multipartite graphs}

The next class of examples is toric models. The Losev-Manin space $\Mznb^{LM}$ (\cite{LM00}) is the closest toric variety to $\Mznb$ among Hassett's moduli spaces $\Mza$ of weighted stable curves (\cite{Has03}). $\Mznb^{LM}$ is constructed as an iterated toric blow-up of $\PP^{n-3}$. Thus if we take an intermediate space of this construction, it gives a toric birational model of $\Mznb$. For any connected graph $G$ of order $n-2$, one can associate a polytope $PG$, a so-called \emph{graph associahedron}, and its corresponding toric variety $X(PG)$ is one of the intermediate spaces. In \cite{dRJR14}, the authors found a characterization of $G$ where $X(PG)$ is indeed one of Hassett's moduli spaces of weighted stable curves. We prove a parallel theorem, in the context of extremal assignments. 

\begin{theorem}
For any connected graph $G$ of order $n-2$, $X(PG)$ is $\Mznb(Z)$ for some extremal assignment $Z$ if and only if $G$ is a complete multipartite graph. 
\end{theorem}

Therefore for almost every graph $G$, $X(PG)$ is not modular. This implies that for the study of contractions of $\Mznb$, studying modular compactifications from extremal assignments is insufficient. 

\subsection{$S_{n}$-invariant models and integer partitions}

We study $S_{n}$-invariant extremal assignments. The natural $S_{n}$-action permuting $n$ labels induces the $S_{n}$-action on $\Mznb$ and the set of extremal assignments of order $n$. For an $S_{n}$-invariant extremal assignment $Z$, on $\Mznb(Z)$, $S_{n}$ acts too, and the contraction $\Mznb \to \Mznb(Z)$ is $S_{n}$-equivariant. Thus an $S_{n}$-invariant extremal assignment $Z$ provides a contraction $\Mznt := \Mznb/S_{n} \to \Mznb(Z)/S_{n}$. 

We prove the following analogue of the structure theorem of ordinary extremal assignments. As one may expect, the theorem can be described in terms of integer partitions, rather than set partitions. However, the result is not a simple restatement, because if we forget labels, then there are nontrivial automorphisms of the underlying graphs even for trees. For the definition of the special family of integer partitions, see Definition \ref{def:specialpartition}.

\begin{theorem}\label{thm:invariantthmintro}
There is a one-to-one correspondence between the set of $S_{n}$-invariant extremal assignments and the set of special families of integer partitions. 
\end{theorem}

In summary, in these three special cases, one may translate a question about extremal assignments or modular contractions of $\Mznb$ into the terms of simpler combinatorial objects. We may summarize the situation as Figure \ref{fig:correspondence}.

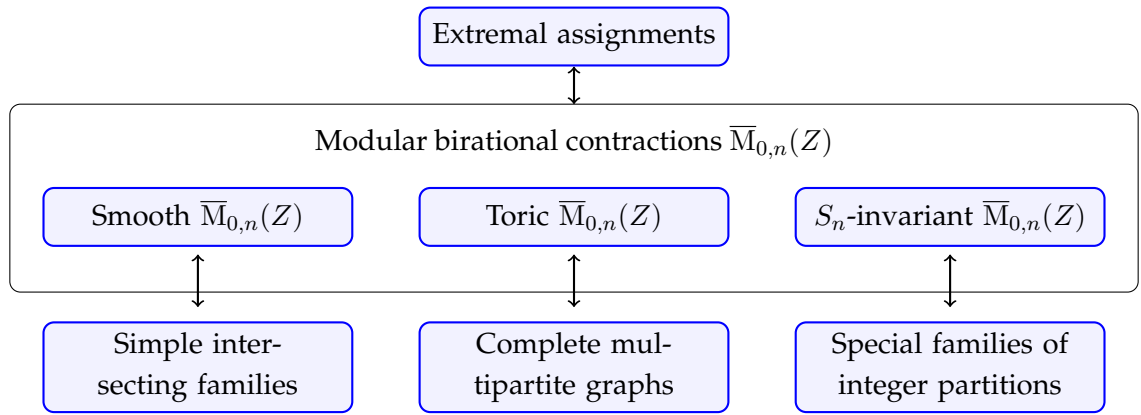
\begin{figure}[!ht]
\begin{tikzpicture}
[auto, block/.style ={rectangle, draw=blue, thick, fill=blue!5, text width=10em,align=center, rounded corners, minimum height=2em}]
\draw (5, 1) node (BC) {Modular birational contractions $\Mznb(Z)$};
\draw [rounded corners] (-2.5,1.5)--(-2.5,-1)--(12.5,-1)--(12.5,1.5)--cycle;
\draw [<->, thick] (0, -0.5) -- (0,-1.2);
\draw [<->, thick] (5, -0.5) -- (5,-1.2);
\draw [<->, thick] (10, -0.5) -- (10,-1.2);
\draw [<->, thick] (5, 1.5) -- (5,2);
\draw (5, 2.4) node[block] (E) {Extremal assignments};
\draw (0, 0) node[block] (S) {Smooth $\Mznb(Z)$};
\draw (5, 0) node[block] (T){Toric $\Mznb(Z)$};
\draw (10, 0) node[block] (I){$S_{n}$-invariant $\Mznb(Z)$};
\draw (0, -2) node[block] (SI) {Simple intersecting families};
\draw (5, -2) node[block] (CM) {Complete multipartite graphs};
\draw (10, -2) node[block] (IP) {Special families of integer partitions};
\end{tikzpicture}
\caption{Three correspondences of combinatorial objects}
\label{fig:correspondence}
\end{figure}

\subsection{Computer programs}

Based on Theorem \ref{thm:structurethmintro} (resp. Theorem \ref{thm:invariantthmintro}), we are able to describe an algorithm for finding the smallest extremal assignment $Z$ in Question \ref{que:fundamentalquestion}. We implemented this algorithm as a program in Sage (\cite{Sage}). It can be found on the website of the first author:
\begin{center}
	\url{http://www.hanbommoon.net/publications/extremal}
\end{center}

\subsection{Classification of contractions and projectivity question}

In the moduli-theoretic viewpoint, there are two known families of birational contractions of $\Mznb$. One is obtained from extremal assignments, as we discussed above. The other family, the so-called family of \emph{Veronese quotients}, is obtained from geometric invariant theory (GIT) (\cite{GJM13, GJMS13}). Constructions of birational contractions in these two viewpoints are very general and exhaustive in some sense, so one may wonder whether all of the projective birational contractions (in the sense of Mori's program or the log minimal model program) of $\Mznb$ are essentially obtained from these two families. 

The structure of the nef cone of $\Mznb$ is generally open for $n \ge 8$, and it is complicated even for small $n$. Thus we focused on the $S_{n}$-invariant nef cone of $\Mznb$, or, equivalently, the nef cone of $\Mznt := \Mznb/S_{n}$. In Section \ref{sec:geography}, we show that for $n \le 8$, all birational models are obtained from extremal assignments and Veronese quotients, but from $n = 9$, some contractions do not come from these two families. 

One interesting result is the following. A priori, the modular contractions $\Mznb(Z)$ from extremal assignments exist in the category of algebraic spaces. However, some of them are not projective varieties. We leave three explicit examples of non-projective contractions. Indeed, one of them is a smooth proper non-projective variety. Thus, by using extremal assignments, we are able to construct many examples of smooth, non-projective proper algebraic varieties. Finally, for smooth extremal assignments, we give a partial criterion for the projectivity (Corollary \ref{cor:projective}, Proposition \ref{prop:existenceweightassignment}), and show that the associated birational contractions are always varieties (Proposition \ref{prop:scheme}). 

\subsection{Further works}

There have been many results (\cite{Man95, Cey09, Moo11, BH11, BM13, BM14}) on the computation of motivic invariants of $\Mznb$, and more generally, for $\Mza$. For instance, their Poincar\'e polynomials were computed in some special cases (\cite{Moo11, BM13, BM14}). It would be great if one could find a concise closed or iterative formula describing the Poincar\'e polynomial for arbitrary weight data. However, it appears that the weight data are not appropriate input data for such a formula, as many different weights give the same space. We expect that the formula may be extracted from the associated smooth extremal assignment, since it determines the moduli space uniquely. We will explore this motivic computation in a forthcoming paper.

\subsection{Structure of the paper}

This paper is organized in the following format. In Section \ref{sec:preliminaries}, we review the definition and basic properties of $\Mznb$, its subvarieties, and dual graphs. We recall the definition of an extremal assignment and its functorial properties in Section \ref{sec:extremalassignment}. Section \ref{sec:example} gives some important classes of examples. In Section \ref{sec:n=5}, we give a complete list of equivalent classes of extremal assignments for $n = 5$, as a warm-up. Section \ref{sec:structuretheorem} is devoted to the proof of the structure theorem (Theorem \ref{thm:structurethm}). In the next several sections, we give applications of this structure theorem. In Section \ref{sec:divisorial}, we give a characterization of extremal assignments giving smooth birational contractions. In Section \ref{sec:toric}, we study toric birational models. $S_{n}$-invariant extremal assignments are discussed in Section \ref{sec:invariant}. By applying this theory, we compute the birational contractions associated to the nef cone of $\Mznt$ for small $n$ in Section \ref{sec:geography}. Finally, we give examples of non-projective models in the last section.

\subsection*{Notation}

\begin{itemize}
\item Let $[n] := \{1, 2, \cdots, n\}$. A \emph{$k$-subset} of $[n]$ is $S \subset [n]$ such that $|S| = k$. 
\item A \emph{labeled graph} is a graph $G$ with a bijective map between the set of end vertices and a given label set. Usually we use $[n]$ as a label set.
\item For any labeled graph $G$, $V(G)$ is the set of internal (thus non-labeled) vertices.
\item An $n$-labeled tree $G$ is called a \emph{star} if every internal vertex is connected to a single vertex $v$. $v$ is called a \emph{central vertex}. 
\item $S(n)$ is the set of all stable $n$-labeled trees.
\item $\cP(n)$ is the set of set partitions of $[n]$. 
\item A \emph{tail} $T$ of a graph (or labeled graph) $G$ is a connected subgraph of $G$ such that there is only one vertex $v \in T$ which is adjacent to $T^{c}$.
\item For a subgraph $T$, the label set $\ell(T)$ is the set of all labels adjacent to $T$. $\ell(v) = \ell(\{v\})$. 
\item We denote an $n$-dimensional weight vector $(a, a, \cdots, a)$ by $a^{n}$. $(a, a, \cdots, a, b)$ is denoted by $(a^{n-1}, b)$, and so on.
\end{itemize}

\subsection*{Acknowledgements}

The first author would like to thank Andreas Blass, Daniel Solt\'esz, and David Swinarski. The first, second, and third authors were partially supported by the Fordham University Summer Undergraduate Research Program.


\section{Preliminaries}\label{sec:preliminaries}

\subsection{The moduli space $\Mznb$}

In this section, we summarize the definition and basic facts of the moduli space $\Mznb$ of stable $n$-pointed rational curves.

\begin{definition}\label{def:stablecurve}
A \emph{stable $n$-pointed rational curve} is a complex $n$-pointed curve $(C, x_{1}, x_{2}, \cdots, x_{n})$ such that 
\begin{enumerate}
\item $C$ is a connected, projective curve of arithmetic genus 0 with at worst nodal singularities;
\item $x_{i}$'s are distinct smooth points on $C$;
\item Every irreducible component of $C$ has at least 3 special points (singular points or marked points).
\end{enumerate}
Let $\Mznb$ be the moduli space of stable $n$-pointed rational curves.
\end{definition}

The space $\Mznb$ is a fine moduli space and it is a smooth projective variety of dimension $n - 3$. If we want to use a specific index set $S$ instead of $[n]$, we denote by $\overline{\mathrm{M}}_{0, S}$. 

By considering loci of singular curves with fixed topological types, we can obtain subvarieties that define a stratification structure on $\Mznb$. 

\begin{definition}
For any $I \subset [n]$ with $2 \le |I| \le n-2$, let $D_{I}$ be the closure of the locus of curves $(C = C_{1} \cup C_{2}, x_{1}, x_{2}, \cdots, x_{n})$ with two irreducible components $C_{1}$ and $C_{2}$ such that $x_{i} \in C_{1}$ if and only if $i \in I$. By definition, $D_{I} = D_{I^{c}}$. $D_{I}$ is called a \emph{boundary divisor}. Let $D = \bigcup_{I}D_{I}$ and $D_{i} = \bigcup_{|I| = i}D_{I}$. 
\end{definition}

Each $D_{I}$ is a divisor of $\Mznb$, and it is isomorphic to $\overline{\mathrm{M}}_{0, I \cup \{p\}} \times \overline{\mathrm{M}}_{0, I^{c} \cup \{q\}}$. The total boundary divisor $D$ is a simple normal crossing divisor. Thus, by taking a nonempty intersection of its irreducible components, we obtain a subvariety with the expected codimension. The intersection $D_{I_{1}} \cap D_{I_{2}} \cap \cdots \cap D_{I_{r}}$ is nonempty if and only if for any two distinct indices $i$ and $j$, $I_{i} \subset I_{j}$ or $I_{i} \cap I_{j} = \emptyset$. The stratification on $\Mznb$ obtained by these nonempty intersections of boundary divisors is called the \emph{boundary stratification}.

\begin{definition}\label{def:Fcurves}
Let $I_{1}\sqcup I_{2}\sqcup I_{3}\sqcup I_{4} = [n]$ be a set partition. For each $i$, we define a `tail' $C_{i}$ as the following. If $|I_{i}| > 1$, then fix a $\PP^{1}$ with $|I_{i}|+1$ distinct fixed points with labels $I_{i}\cup \{p_{i}\}$. Let $C_{i}$ be the pointed curve. If $|I_{i}| = 1$, $C_{i} = \{p_{i}\}$ is a point. Construct $(C, x_{1}, x_{2}, \cdots, x_{n}) \in \Mznb$ by gluing four fixed tails $C_{i}$ to a four pointed curve $(\PP^{1}, q_{1}, q_{2}, q_{3}, q_{4})$ along $p_{i}$ and $q_{i}$. An \emph{F-curve} $F_{I_{1}, I_{2}, I_{3}, I_{4}} \subset \Mznb$ is obtained by varying the cross ratio of $q_{j}$'s. On $\Mznt := \Mznb/S_{n}$, we define an F-curve $F_{i_{1}, i_{2}, i_{3}, i_{4}}$ as any F-curve $F_{I_{1}, I_{2}, I_{3}, I_{4}}$ so that $|I_{j}| = i_{j}$. 
\end{definition}

An F-curve $F_{I_{1}, I_{2}, I_{3}, I_{4}}$ is numerically equivalent to any 1-dimensional nonempty intersection of boundaries 
\[
	\left(\bigcap_{|I_{i}| > 1}D_{I_{i}}\right) 
	\cap D_{J_{1}} \cap D_{J_{2}} \cap \cdots.
\]

The intersection number of a boundary divisor and an F-curve is given by the following lemma.

\begin{lemma}\label{lem:intersection}\cite[Lemma 4.3]{KM13}.
\[
	F_{I_{1}, I_{2}, I_{3}, I_{4}} \cdot D_{J} = 
	\begin{cases}
	1, & J = I_{i}\cup I_{j} \mbox{ for two $i, j$},\\
	-1, & J = I_{i},\\
	0, & \mbox{otherwise.}\end{cases}
\]
\end{lemma}

Boundary divisors generate $\mathrm{N}^{1}(\Mznb)$, and F-curves generate $\mathrm{N}_{1}(\Mznb)$.

\subsection{Graph theoretic notation}

The boundary stratification of $\Mznb$ can be described in terms of stable labeled trees. In this section, we briefly review their combinatorics. 

\begin{definition}
A \emph{stable $n$-labeled tree} is a finite graph with $n$ leaves $\{1, 2, \cdots,n\}$ satisfying:
\begin{enumerate}
\item The underlying graph is a tree (i.e., there is no circuit).
\item Every (internal) vertex has at least 3 adjacent edges.
\end{enumerate}
Let $S(n)$ be the set of all stable $n$-labeled trees. Let $S_{k}(n) \subset S(n)$ be the subset of stable $n$-labeled trees with precisely $k$ vertices. It is straightforward to check that 
\[
	S(n) = \bigcup_{k=1}^{n-2}S_{k}(n).
\]
\end{definition}

If there is no further explanation, a vertex of a labeled tree is always an internal vertex. 

We can impose a natural partially ordered set structure on $S(n)$. 

\begin{definition}
For $G, G' \in S(n)$, we say $G'$ is a contraction of $G$ (denoted by $G \rightsquigarrow G'$) if we can make $G'$ by collapsing some vertices and edges connecting them to a vertex. For a contraction $\pi : G \rightsquigarrow G'$ and $v' \in V(G')$, by abuse of notation we denote $V(G) \to V(G')$ by $\pi$. If $\{v_{1}, v_{2}, \cdots, v_{k}\} = \pi^{-1}(v')$, we write $\{v_{1}, v_{2}, \cdots, v_{k}\} \rightsquigarrow v$. 

Let $G \le G'$ if $G \rightsquigarrow G'$. Then $S(n)$ is a partially ordered set.
\end{definition}

\begin{example}
Figure \ref{fig:contraction} shows two stable 6-labeled trees such that $G\leadsto G'$ and $\{v_0,v_1\}\leadsto v'$.
\begin{figure}[!ht]
\begin{tikzpicture}[scale=0.3]
\draw[line width = 1 pt] (5, 7) -- (5, 9);
\draw[line width = 1 pt] (5, 9) -- (3, 11);
\draw[line width = 1 pt] (5, 9) -- (7, 11);
\draw[line width = 1 pt] (5, 5) -- (7, 3);
\draw[line width = 1 pt] (5, 5) -- (5, 7);
\draw[line width = 1 pt] (5, 5) -- (3, 3);
\draw[line width = 1 pt] (5, 7) -- (3, 7);
\draw[line width = 1 pt] (5, 7) -- (7, 7);
\fill (5, 5) circle (6pt);
\fill (5, 7) circle (6pt);
\node (v) at (5, 4) {$v_{0}$};
\node (1) at (3,2) {$1$};
\node (2) at (7,2) {$2$};
\node (3) at (2,7) {$3$};
\node (4) at (8,7) {$4$};
\node (5) at (3,12) {$5$};
\node (6) at (7,12) {$6$};
\node (v) at (4.5, 8) {$v_{1}$};
\node (G) at (5, 1) {$G$};
\node (a) at (9.75, 6) {$\leadsto$};
\end{tikzpicture}
\begin{tikzpicture}[scale=0.3]
\draw[line width = 1 pt] (5, 7) -- (5, 9);
\draw[line width = 1 pt] (5, 9) -- (3, 11);
\draw[line width = 1 pt] (5, 9) -- (7, 11);
\draw[line width = 1 pt] (5, 6) -- (3, 6);
\draw[line width = 1 pt] (5, 6) -- (7, 6);
\draw[line width = 1 pt] (5, 6) -- (6.5, 3.5);
\draw[line width = 1 pt] (5, 6) -- (5, 7);
\draw[line width = 1 pt] (5, 6) -- (3.5, 3.5);
\fill (5, 6) circle (6pt);
\node (v) at (5, 4.8) {$v'$};
\node (G) at (5, 1) {$G'$};
\node (1) at (3,2) {$1$};
\node (2) at (7,2) {$2$};
\node (3) at (2.5,6) {$3$};
\node (4) at (7.5,6) {$4$};
\node (5) at (3,12) {$5$};
\node (6) at (7,12) {$6$};
\end{tikzpicture}
\caption{An example of contraction}
\label{fig:contraction}
\end{figure}
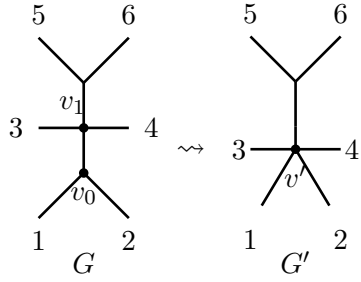
\end{example}

\begin{definition}
For $X := (C, x_{1}, x_{2}, \cdots, x_{n}) \in \Mznb$, the \emph{dual graph} $G$ is a stable $n$-labeled tree constructed as the following. The set $V(G)$ of vertices is the set of irreducible components of $C$. Two vertices $v, w$ are connected if two corresponding irreducible components meet. We put a label $i$ on a vertex $v$ if $x_{i}$ lies on the corresponding irreducible component. The condition (3) of Definition \ref{def:stablecurve} implies the stability of $G$.
\end{definition}

\begin{example}
Figure \ref{fig:dual} shows a curve in $\overline{\mathrm{M}}_{0,6}$ and its dual graph.
\begin{figure}[!ht]
\begin{tikzpicture}[scale=0.45]
\draw[line width = 1 pt] (3, 5) -- (8, 5);
\draw[line width = 1 pt] (3.5, 4.5) -- (3.5, 8);
\draw[line width = 1 pt] (7.5, 4.5) -- (7.5, 8);
\fill (3.5, 6.25) circle (6pt);
\fill (3.5, 7.5) circle (6pt);
\fill (7.5, 6.25) circle (6pt);
\fill (7.5, 7.5) circle (6pt);
\fill (4.75, 5) circle (6pt);
\fill (6.5, 5) circle (6pt);
\node (1) at (2.5,6.25) {$1$};
\node (2) at (2.5,7.5) {$2$};
\node (3) at (4.75,4) {$3$};
\node (4) at (6.5,4) {$4$};
\node (5) at (8.5,6.25) {$5$};
\node (6) at (8.5,7.5) {$6$};
\end{tikzpicture}
\qquad
\begin{tikzpicture}[scale=0.3]
\draw[line width = 1 pt] (5, 7) -- (5, 9);
\draw[line width = 1 pt] (5, 9) -- (3, 11);
\draw[line width = 1 pt] (5, 9) -- (7, 11);
\draw[line width = 1 pt] (5, 5) -- (7, 3);
\draw[line width = 1 pt] (5, 5) -- (5, 7);
\draw[line width = 1 pt] (5, 5) -- (3, 3);
\draw[line width = 1 pt] (5, 7) -- (3, 7);
\draw[line width = 1 pt] (5, 7) -- (7, 7);
\node (1) at (3,2) {$1$};
\node (2) at (7,2) {$2$};
\node (3) at (2.5,7) {$3$};
\node (4) at (7.5,7) {$4$};
\node (5) at (3,12) {$5$};
\node (6) at (7,12) {$6$};
\end{tikzpicture}
\caption{A stable 6-pointed rational curve and its dual graph}
\label{fig:dual}
\end{figure}
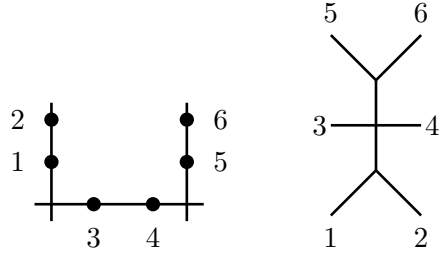
\end{example}

\begin{example}
Suppose that $G$ is a tree with two vertices, and the set of labels on a vertex is $I\subset [n]$. Then the corresponding curve $X = (C_{1}, x_{1}, x_{2}, \cdots, x_{n})$ is on $D_{I}$. 
\end{example}

For $G \in S(n)$, let $D_{G} \subset \Mznb$ be the locus of curves whose dual graph is $G$. Then $D_{G}$ is a smooth locally-closed subvariety of $\Mznb$ that is an intersection of several boundary divisors. Also $G \le G'$ if and only if $D_{G} \subset \overline{D_{G'}}$.


\section{Extremal assignments and birational contractions of $\Mznb$}\label{sec:extremalassignment}

In \cite{Smy13}, Smyth introduced a combinatorial object, a so-called extremal assignment, which defines a contraction of $\Mznb$. In this section, we give its definition and basic properties. 

\begin{definition}\label{def:assignment}
\begin{enumerate}
\item An \emph{assignment} $\overline{Z}$ on $T \subset S(n)$ is a rule assigning for each $G \in T$ a subset $\overline{Z}(G) \subset V(G)$. 
\item An assignment $Z$ on $S(n)$ is called an \emph{extremal assignment of order $n$} if:
\begin{enumerate}
\item $Z(G) \ne V(G)$ for any $G \in S(n)$;
\item If $G \rightsquigarrow G'$ and $\{v_{1}, v_{2}, \cdots, v_{k}\} \rightsquigarrow v'$, then $v' \in G' \Leftrightarrow v_{1}, v_{2}, \cdots, v_{k} \in G$. 
\end{enumerate}
\end{enumerate}
\end{definition}

\begin{remark}
In \cite[Definition 1.5]{Smy13}, Smyth defines an extremal assignment for arbitrary genus $g$ graphs. It is straightforward to see that his condition (2) is obvious for the $g = 0$ case, because any stable $n$-labeled tree has a trivial automorphism group.
\end{remark}

We define several natural operations. 

\begin{definition}
An extremal assignment $Z'$ is called a \emph{sub extremal assignment} of $Z$ if $Z'(G) \subset Z(G)$ for every $G \in S(n)$. We denote by $Z' \subset Z$.
\end{definition}

\begin{definition}
Let $Z_{1}, \cdots, Z_{k}$ be extremal assignments. The \emph{intersection} $Z := \bigcap_{i=1}^{k}Z_{i}$ is defined by $Z(G) := \bigcap_{i=1}^{k}Z_{i}(G)$. $Z$ is also an extremal assignment. 
\end{definition}

However, in general the \emph{union} $Z := \bigcup_{i=1}^{k}Z_{i}$, which is defined by $Z(G) = \bigcup_{i=1}^{k}Z_{i}(G)$, is \emph{not} an extremal assignment. 

\begin{example}
Let $G\in S_2(n)$ have two vertices, $v_1$ and $v_2$. Let $Z_1$ and $Z_2$ be two extremal assignments with $v_1\in Z_1(G)$ and $v_2\in Z_2(G)$. Then $Z_1\cup Z_2$ is not an extremal assignment, since $Z(G)=V(G)$.
\end{example}

For a classification of modular compactifications of $\Mzn$, extremal assignments have a central role. A \emph{modular compactification} of $\Mzn$ is a compactification that can be interpreted as a moduli space of a certain type of pointed curves. More precisely, it is defined as a proper open substack $\cM$ of the stack of all $n$-pointed genus 0 curves containing $\Mzn$. 

By using an extremal assignment $Z$, we can define a modular compactification $\Mznb(Z)$ as follows. Since there is a one-to-one correspondence between the set of irreducible components on $(C, x_{1}, \cdots, x_{n}) \in \Mznb$ and $V(G)$ for its dual graph $G$, $Z(G)$ defines a subset of irreducible components of $C$. We will denote the subset by $Z(C)$. 

\begin{definition}[\protect{\cite[Definition 1.8]{Smy13}}]\label{def:Zstability}
Let $Z$ be an extremal assignment. A smoothable $n$-pointed curve $(C, x_{1}, x_{2}, \cdots, x_{n})$ of arithmetic genus 0 is \emph{$Z$-stable} if there exists $(C^{s}, x_{1}^{s}, \cdots, x_{n}^{s}) \in \Mznb$ and a map $\phi : (C^{s}, x_{1}^{s}, \cdots, x_{n}^{s}) \to (C, x_{1}, \cdots, x_{n})$ satisfying
\begin{enumerate}
\item $\phi$ is surjective with connected fibers;
\item $\phi|_{C^{s} - Z(C^{s})}$ is an isomorphism;
\item For a connected component $Z'$ of $Z(C^{s})$, $\phi(Z')$ is a single point.
\end{enumerate}
We can naturally define the notion of flat families of $Z$-stable curves. Let $\Mznb(Z)$ be the moduli stack of $Z$-stable curves. 
\end{definition}

\begin{theorem}[\protect{\cite[Theorem 1.9, Theorem 1.21]{Smy13}}]\label{thm:extassignmentmodel}
Let $Z$ be an extremal assignment. Then $\Mznb(Z)$ is a modular compactification of $\Mzn$. Conversely, for any modular compactification $\cM$ of $\Mzn$, there exists an extremal assignment $Z$ such that $\cM \cong \Mznb(Z)$. The compactification $\Mznb(Z)$ exists in the category of algebraic spaces. 
\end{theorem}

\begin{remark}
Smyth proved a more general result for arbitrary genus $g$. In the case of genus 0, which will be discussed in this paper, we have the following result as well.
\end{remark}

\begin{proposition}\label{prop:functoriality}
\begin{enumerate}
\item For any extremal assignment $Z$, there is a birational contraction $\pi_{Z} : \Mznb \to \Mznb(Z)$ which has connected fibers. 
\item More generally, if $Z_{1} \subset Z_{2}$, there is a birational map $\pi_{Z_{1}, Z_{2}} : \Mznb(Z_{1}) \to \Mznb(Z_{2})$. Furthermore, for any three extremal assignments $Z_{1} \subset Z_{2} \subset Z_{3}$, there is a commutative diagram
\[
	\xymatrix{\Mznb(Z_{1}) \ar[rr]^{\pi_{Z_{1},Z_{3}}} 
	\ar[rd]_{\pi_{Z_{1}, Z_{2}}}&& \Mznb(Z_{3})\\
	& \Mznb(Z_{2}) \ar[ru]_{\pi_{Z_{2},Z_{3}}}.}
\]
\end{enumerate}
\end{proposition}

\begin{proof}
For any family $(\rho : \cC \to S, \sigma_{1}, \cdots, \sigma_{n} : S \to \cC)$ of stable $n$-pointed rational curves, by contracting irreducible components assigned by the data $Z$, we obtain a family of $Z$-stable curves $\cC' \to S$. At the image of a contracted component, a fiber of $\cC'$ has a singularity of arithmetic genus zero. From the classification of arithmetic genus zero singularities (\cite[Lemma 1.17]{Smy13}), it is a multinodal singularity, which is locally the union of coordinate axes in $\CC^{k}$. In particular, there is no nontrivial moduli for this singularity. Therefore, the contracted curve is defined uniquely, so we have a well-defined map $\pi_{Z} : \Mznb \to \Mznb(Z)$. This proves (1). 

If $Z_{1} \subset Z_{2}$, it is straightforward to check that for any $x \in \Mznb(Z_{1})$, the fiber $\pi_{Z_{1}}^{-1}(x) \subset \Mznb$ is contracted to a single point in $\Mznb(Z_{2})$ by $\pi_{Z_{2}}$. By the rigidity lemma, the map $\pi_{Z_{2}}$ factors through $\Mznb(Z_{1})$. We will denote this map by $\pi_{Z_{1}, Z_{2}}$. The commutativity is obvious from their set theoretic descriptions. 
\end{proof}

The contracted locus of the reduction map $\pi_{Z_{1}, Z_{2}}$ can be described in terms of restricted extremal assignments.

\begin{definition}\label{def:restriction}
Fix an extremal assignment $Z$. Let $G \in S(n)$ and let $v \in V(G) \setminus Z(G)$. Let $E_{v}$ be the set of edges adjacent to $v$. Then with new label set $\ell(v) \cup E_{v}$, we can define the \emph{restriction} $Z|_{v}$ of $Z$ as the following. For any $H \in S(|\ell(v)| + |E_{v}|)$, we can find a degeneration $\pi : G' \rightsquigarrow G$ such that $\pi^{-1}(v)$ in $G'$ is isomorphic to $H$. Then for $w \in V(H)$, set $w \in Z|_{v}(H) \Leftrightarrow w \in Z(G')$. If $G \in S_{2}(n)$ with two vertices $v, w$ and the label set for $v$ (resp. $w$) is $B$ (resp. $B^{c}$), then we write $Z^{B} := Z|_{v}$ (resp. $Z^{B^{c}} := Z|_{w}$). 
\end{definition}

\begin{remark}\label{rem:contractedlocus}
For a graph $G \in S(n)$, recall that $D_{G} \subset \Mznb$ is the locus of curves whose dual graph is $G$. We denote its image on $\Mznb(Z_{1})$ with the same notation. Then 
\[
	\overline{D_{G}} \cong \prod_{v \in V(G) \setminus Z_{1}(G)}
	\overline{\mathrm{M}}_{0, |\ell(v)|+|E_{v}|}(Z_{1}|_{v}).
\]
The image $\pi_{Z_{1}, Z_{2}}(\overline{D_{G}})$ is obtained by forgetting newly assigned components. More precisely, 
\[
	\pi_{Z_{1}, Z_{2}}(\overline{D_{G}}) \cong \prod_{v \in V(G) \setminus Z_{2}(G)}
	\overline{\mathrm{M}}_{0, |\ell(v)|+|E_{v}|}(Z_{2}|_{v}).
\]
\end{remark}

\begin{example}
Let $Z$ be an extremal assignment that assigns $v \in V(G)$ for a dual graph $G$ in Figure \ref{fig:contractionexample}. Let $(C, x_{1}, x_{2}, \cdots, x_{6}) \in \Mznb$ be a stable 6-pointed curve whose dual graph is $G$. Then the central component is assigned, thus the contraction $(C', x_{1}', x_{2}', \cdots, x_{6}')$ is a $Z$-stable curve. See Figure \ref{fig:contractionexample}. In general, two or more marked points may collide, and some marked points may lie on a singular point. If an assigned vertex has three or more adjacent internal vertices, then the contracted curve has a multinodal singularity.
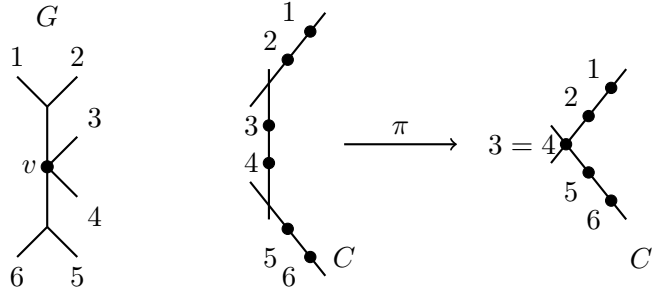
\begin{figure}[!ht]
\begin{tikzpicture}[scale=0.4]
\draw [thick] (0,2) -- (0,6);
\draw [thick] (0,4) -- (1,5);
\node [above right] at (1,5) {$3$};
\draw [thick] (0,4) -- (1,3);
\node [below right] at (1,3) {$4$};
\draw [thick] (0,6) -- (1,7);
\node [above] at (1,7) {$2$};
\draw [thick] (0,6) -- (-1,7);
\node [above] at (-1,7) {$1$};
\draw [thick] (0,2) -- (-1,1);
\node [below] at (-1,1) {$6$};
\draw [thick] (0,2) -- (1,1);
\node [below] at (1,1) {$5$};
\draw [fill] (0,4) circle [radius=0.2];
\node [left] at (0, 4) {$v$};
\node at (0,9) {$G$};
\end{tikzpicture}
\qquad\qquad
\begin{tikzpicture}[scale=0.5]
\draw [thick] (-4,-2) -- (-4,2);
\draw [thick] (-4.5,1) -- (-2.5,3.5);
\draw [thick] (-4.5,-1) -- (-2.5,-3.5);
\draw [fill] (-3.5,2.25) circle [radius=0.15];
\node [above left] at (-3.5,2.25) {$2$};
\draw [fill] (-2.9,3) circle [radius=0.15];
\node [above left] at (-3,3) {$1$};
\draw [fill] (-4,0.5) circle [radius=0.15];
\node [left] at (-4,0.5) {$3$};
\draw [fill] (-4,-0.5) circle [radius=0.15];
\node [left] at (-4,-0.5) {$4$};
\draw [fill] (-3.5,-2.25) circle [radius=0.15];
\node [below left] at (-3.5,-2.5) {$5$};
\draw [fill] (-2.9,-3) circle [radius=0.15];
\node [below left] at (-3,-3) {$6$};
\draw [thick] [->] (-2,0) -- (1, 0);
\draw [thick] (3.5,-.5) -- (5.5,2);
\draw [thick] (3.5,.5) -- (5.5,-2);
\node [left] at (3.9,0) {$3=4$};
\draw [fill] (4.5,0.75) circle [radius=0.15];
\node [above left] at (4.5,0.75) {$2$};
\draw [fill] (5.1,1.5) circle [radius=0.15];
\node [above left] at (5.1,1.5) {$1$};
\draw [fill] (4.5, -0.75) circle [radius=0.15];
\node [below left] at (4.5,-0.75) {$5$};
\draw [fill] (5.1,-1.5) circle [radius=0.15];
\node [below left] at (5.1,-1.5) {$6$};
\draw [fill] (3.9,0) circle [radius=0.15];
\node at (-2, -3) {$C$};
\node at (6,-3) {$C'$};
\node [above] at (-0.5,0) {$\pi$};
\end{tikzpicture}
\caption{An example of contraction map}
\label{fig:contractionexample}
\end{figure}
\end{example}

Sometimes for two extremal assignments $Z_{1}, Z_{2}$, the corresponding contractions $\Mznb(Z_{1})$ and $\Mznb(Z_{2})$ are bijective. We would like to identify these cases.

\begin{definition}
Let $G \in S(n)$ and $W(G) \subset V(G)$. A vertex $v \in W(G)$ is called \emph{isolated} if none of the vertices adjacent to $v$ is in $W(G)$. 
\end{definition}

\begin{definition}
We say two extremal assignments $Z_{1}$ and $Z_{2}$ of order $n$ are \emph{equivalent} (denoted by $Z_{1} \sim Z_{2}$) if for any $G \in S(n)$, both $Z_{1}(G)\setminus Z_{2}(G)$ and $Z_{2}(G) \setminus Z_{1}(G)$ are unions of isolated 3-valenced points. 
\end{definition}

\begin{lemma}
The equivalence of extremal assignments is an equivalence relation. 
\end{lemma}

\begin{proof}
From the definition, $Z \sim Z$ and $Z_{1} \sim Z_{2} \Rightarrow Z_{2} \sim Z_{1}$ are obvious. Suppose that $Z_{1} \sim Z_{2}$ and $Z_{2} \sim Z_{3}$. Suppose that $v_{1}, v_{2}$ are two adjacent 3-valenced vertices in $Z_{1}(G) \setminus Z_{3}(G)$. Let $G' \in S(n)$ be the graph obtained by contracting $v_{1}, v_{2}$ to a 4-valenced vertex $w$ from $G$. Then $w \in Z_{1}(G')$. Since $w$ is a 4-valenced vertex, $Z_{1}(G') = Z_{2}(G') = Z_{3}(G')$. From $w \in Z_{3}(G')$, $v_{1}, v_{2} \in Z_{3}(G)$, so there arises a contradiction. So only one of $v_{1}$ and $v_{2}$ can be in $Z_{1}(G) \setminus Z_{3}(G)$ and $Z_{1}(G) \setminus Z_{3}(G)$ is a union of isolated 3-valenced points. $Z_{3}(G) \setminus Z_{1}(G)$ is the same.
\end{proof}

The below proposition justifies the definition of the equivalence relation. 

\begin{proposition}\label{prop:equivalentsamemodel}
Let $Z_{1}, Z_{2}$ be two extremal assignments. If $Z_{1}\sim Z_{2}$, then there is a homeomorphism between $\Mznb(Z_{1})$ and $\Mznb(Z_{2})$, which preserves $\Mzn$. In this case, if $X^{\nu}$ is the normalization of $X$, then $\Mznb(Z_{1})^{\nu} \cong \Mznb(Z_{2})^{\nu}$.
\end{proposition}

\begin{proof}
Suppose that $Z_{1} \sim Z_{2}$. Let $Z_{3} := Z_{1} \cap Z_{2}$. Then it is straightforward to check that $Z_{1} \sim Z_{3} \sim Z_{2}$. By Proposition \ref{prop:functoriality}, we have two morphisms 
\begin{equation}\label{eqn:decomposition}
	\xymatrix{&\Mznb(Z_{3}) \ar[rd]^{\pi_{Z_{3}, Z_{1}}} 
	\ar[ld]_{\pi_{Z_{3}, Z_{2}}}\\
	\Mznb(Z_{1}) && \Mznb(Z_{2}).}
\end{equation}
These are surjective since they are dominant morphisms between proper varieties. For $X := (C, x_{1}, x_{2}, \cdots, x_{n}) \in \Mznb(Z_{3})$, two maps $\pi_{Z_{3},Z_{1}}$ and $\pi_{Z_{3},Z_{2}}$ contract 3-pointed components of $X$ only. Since there is no positive-dimensional moduli on 3-pointed rational curves, $\pi_{Z_{3}, Z_{1}}$ and $\pi_{Z_{3}, Z_{2}}$ are bijective. 

Furthermore, since $\pi_{Z_{i}} : \Mznb \to \Mznb(Z_{i})$ is a morphism from a smooth variety with connected fibers, the normalization map $\Mznb(Z_{i})^{\nu} \to \Mznb(Z_{i})$ is bijective. Therefore $\Mznb(Z_{3})^{\nu} \to \Mznb(Z_{i})^{\nu}$ is a bijective map between normal varieties, so it is an isomorphism.
\end{proof}

If $Z_{2} \subset Z_{1}$, unless $Z_{1}(G)\setminus Z_{2}(G)$ is a disjoint union of isolated 3 valenced vertices, there is a further contraction on the moduli space. Thus $\pi_{Z_{2}, Z_{1}} : \Mznb(Z_{2}) \to \Mznb(Z_{1})$ is not injective. Thus we have the following partial converse.

\begin{proposition}\label{prop:bijectiveimpliesequivalence}
Suppose that $Z_{2} \subset Z_{1}$ and the reduction map $\pi_{Z_{2}, Z_{1}} : \Mznb(Z_{2}) \to \Mznb(Z_{1})$ is bijective. Then $Z_{1} \sim Z_{2}$. 
\end{proposition}


\section{Examples}\label{sec:example}

Before investigating their combinatorics, we leave some important classes of examples of extremal assignments. 

\subsection{Trivial assignment}

If $Z(G) = \emptyset$ for every $G \in S(n)$, then there is no contracted component on any $(C, x_{1}, x_{2}, \cdots, x_{n}) \in \Mznb$. Therefore $\Mznb(Z) = \Mznb$. 

\subsection{Weight assignments}\label{ssec:weightassignment}

In \cite{Has03}, Hassett introduced a large family of modular contractions of $\Mznb$.

\begin{definition}
Fix weight data $A = (a_{1}, a_{2}, \cdots, a_{n})$, that is, a sequence of $a_{i} \in \QQ$ such that $0 < a_{i} \le 1$ and $\sum a_{i} > 2$. A pointed curve $(C, p_{1}, p_{2}, \cdots, p_{n})$ is \emph{$A$-stable} if 
\begin{enumerate}
\item $C$ is a connected, projective curve of arithmetic genus zero with at worst nodal singularities;
\item $p_{i}$'s are all smooth points on $C$;
\item For any smooth point $x \in C$, $\sum_{p_{i} = x}a_{i} \le 1$;
\item For every irreducible component $K$ of $C$, the number of singular points plus $\sum_{p_{i} \in K}a_{i}$ is greater than 2. 
\end{enumerate}
Let $\Mza$ be the moduli space of $A$-stable curves. 
\end{definition}

\begin{remark}
It is straightforwrad to check that $\Mza = \Mznb$ if $A = (1, 1, \cdots, 1)$. 
\end{remark}

\begin{theorem}[\protect{\cite[Theorem 4.1]{Has03}}]
For any weight data $A = (a_{1}, a_{2}, \cdots, a_{n})$, 
\begin{enumerate}
\item $\Mza$ is a smooth projective variety;
\item If $B = (b_{1}, b_{2}, \cdots, b_{n})$ is some other weight data such that $b_{i} \ge a_{i}$, there is a reduction map $\rho_{B, A} : \overline{\mathrm{M}}_{0, B} \to \Mza$. 
\end{enumerate}
\end{theorem}

Any $\Mza$ can be described as $\Mznb(Z)$ for some $Z$. 

\begin{definition}\label{def:Hassettextremalassignment}
Let $A$ be a collection of weight data. Define an extremal assignment $Z_{A}$ as follows. For $G \in S(n)$, 
\[
	Z_{A}(G) = \{v \in G\;|\; \exists \; \mbox{ a tail } T \subset G, 
	v \in T, \sum_{i \in T}a_{i} \le 1\}.
\]
We will call $Z_{A}$ a \emph{weight assignment} from weight data $A$. 
\end{definition}

It was shown that $\Mza = \Mznb(Z_{A})$ (\cite[Example 1.11]{Smy13}). 

\subsection{Kontsevich-Boggi compactification}

Define $Z_{B}(G)$ as the set of all vertices without any labels on $G$. Then $Z$ is an extremal assignment. The corresponding birational model $\Mznb(Z_{B})$ is called \emph{Kontsevich-Boggi compactification}, which was introduced in \cite{Bog99} with completely different terminology. 

\subsection{GIT compactifications}\label{ssec:GITmodels}

All of the above examples were unified and generalized in \cite{GJM13}. There is a broader family of \emph{Veronese quotients} or \emph{GIT compactifications}. This is also a generalization of $(\PP^{1})^{n}\git \SL_{2}$, which is a natural compactification of $\Mzn$ from a viewpoint of invariant theory. It contains $\Mznb$ itself, Hassett's moduli spaces of weighted pointed curves, and Kontevich-Boggi compactification. 

Fix a positive integer $d$. We fix $(\gamma, c_{1}, c_{2}, \cdots, c_{n}) \in \QQ^{n+1}$ such that $(d-1)\gamma + \sum c_{i} = d + 1$ and $0 \le \gamma < 1$ and $0 < c_{i} < 1$ for all $i$. For $G \in S(n)$ and its tail $T$, $c_{T} := \sum_{i \in T}c_{i}$. Assume that $\frac{c_{T}-1}{1-\gamma}$ is not an integer for any $T$ (Remark \ref{rem:semistable}). Define a degree of $T$ as 
\begin{equation}\label{eqn:degreefunction}
	\sigma(T) = \begin{cases}
	\lceil \frac{c_{T}-1}{1-\gamma}\rceil, & 
	\mbox{if } 1 < c_{T} < c_{[n]}-1,\\
	0, & \mbox{if } c_{T} < 1,\\
	d, & \mbox{if } c_{T} > c - 1.\end{cases}
\end{equation}
Note that any vertex $v \in V(G)$ is a difference of two tails $T_{1} \subset T_{2}$, so we may define $\sigma(v) := \sigma(T_{2}) - \sigma(T_{1})$. Two combinatorial results in \cite[Section 3]{GJM13} are:
\begin{enumerate}
\item The degree function $\sigma : V(G) \to \ZZ$ is well-defined;
\item $\sum_{v \in V(G)}\sigma(v) = d$.
\end{enumerate}

\begin{definition}
Define an assignment $Z_{\gamma, \vec{c}}$ as the following:
\[
	Z_{\gamma, \vec{c}}(G) = \{v \in V(G)\;|\; \sigma(v) = 0\}
\]
Then it is indeed an extremal assignment (\cite[Proposition 5.7]{GJM13}). 
\end{definition}

\begin{theorem}[\protect{\cite[Theorem 1.1, Theorem 5.2]{GJM13}}]
For any parameters $d \in \NN$ and $(\gamma, \vec{c}) \in \QQ_{> 0}^{n+1}$ such that $(d-1)\gamma + \sum c_{i} = d + 1$, there is a projective variety $V_{\gamma, \vec{c}}^{d}$ with a birational contraction morphism $\Mznb \to V_{\gamma, \vec{c}}^{d}$ with connected fibers. The space $V_{\gamma, \vec{c}}^{d}$ is constructed as a GIT quotient $U_{d, n}\git_{\gamma, \vec{c}}\SL_{d+1}$ where $U_{d, n}$ is an incidence variety in $\mathrm{Chow}_{1,d}(\PP^{d}) \times (\PP^{d})^{n}$. Furthermore, if $\frac{c_{T} -1}{1-\gamma}$ is not an integer for every tail $T$ of $G \in S(n)$, $V_{\gamma, \vec{c}}^{d} \cong \Mznb(Z_{\gamma, \vec{c}})$.
\end{theorem}

\begin{remark}\label{rem:semistable}
If $\frac{c_{T} - 1}{1-\gamma}$ is an integer for some tail $T$, then $\sigma$ is not well-defined. However, still $V_{\gamma, \vec{c}}^{d}$ has a moduli theoretic meaning in a weak sense, namely, as a good moduli space of an Artin stack. For the details, see \cite[Section 6.3]{GJM13}.
\end{remark}


\section{$n = 5$ case}\label{sec:n=5}

As a first nontrivial case, in this section we classify all equivalent classes of extremal assignments of order 5. Note that there are three topological types of 5-labeled trees. 
\begin{table}[!ht]
\begin{tabular}{|>{\centering\arraybackslash}m{2.2cm}|>{\centering\arraybackslash}m{10ex}|>{\centering\arraybackslash}m{10ex}|>{\centering\arraybackslash}m{10ex}|}\hline
type & I & II & III \\ \hline 
graph & 
\begin{tikzpicture}[scale=0.3]
\draw[line width = 1 pt] (5, 5) -- (2, 4);
\draw[line width = 1 pt] (5, 5) -- (8, 4);
\draw[line width = 1 pt] (5, 5) -- (3, 8);
\draw[line width = 1 pt] (5, 5) -- (7, 8);
\draw[line width = 1 pt] (5, 5) -- (5, 1);
\fill (5, 5) circle (10pt);
\end{tikzpicture} & 
\begin{tikzpicture}[scale=0.3]
\draw[line width = 1 pt] (5, 3) -- (2, 1);
\draw[line width = 1 pt] (5, 3) -- (8, 1);
\draw[line width = 1 pt] (5, 3) -- (5, 6);
\draw[line width = 1 pt] (5, 6) -- (8, 7);
\draw[line width = 1 pt] (5, 6) -- (2, 7);
\draw[line width = 1 pt] (5, 6) -- (5, 9);
\fill (5, 3) circle (10pt);
\fill (5, 6) circle (10pt);
\end{tikzpicture} & 
\begin{tikzpicture}[scale=0.3]
\draw[line width = 1 pt] (5, 3) -- (2, 1);
\draw[line width = 1 pt] (5, 3) -- (8, 1);
\draw[line width = 1 pt] (5, 4.5) -- (8, 4.5);
\draw[line width = 1 pt] (5, 3) -- (5, 6);
\draw[line width = 1 pt] (5, 6) -- (8, 8);
\draw[line width = 1 pt] (5, 6) -- (2, 8);
\fill (5, 6) circle (10pt);
\fill (5, 4.5) circle (10pt);
\fill (5, 3) circle (10pt);
\end{tikzpicture}
\\ \hline
contained in & $S_{1}(5)$ & $S_{2}(5)$ & $S_{3}(5)$\\ \hline
\end{tabular}
\medskip
\caption{topological types of dual graphs in $S(5)$}
\end{table}

\begin{lemma}\label{lem:assignmentandS2}
Suppose that $Z_{1}$ and $Z_{2}$ are two extremal assignments of order 5 and both $Z_{1}(G)\setminus Z_{2}(G)$ and $Z_{2}(G)\setminus Z_{1}(G)$ are empty or isolated 3-valenced for all $G \in S_{2}(5)$. Then $Z_{1} \sim Z_{2}$. 
\end{lemma}

\begin{proof}
Suppose not. If $Z_{1} \not\sim Z_{2}$, then there is a graph $G$ with two adjacent 3-valenced vertices $v_{1}, v_{2}$ in $Z_{1}(G) \setminus Z_{2}(G)$. Then $G \in S_{3}(5)$ and there is $G' \in S_{2}(5)$ such that $G \rightsquigarrow G'$ and $\{v_{1}, v_{2}\} \rightsquigarrow v'$. Then $\ell(v') = 4$ and $v' \in Z_{1}(G') \setminus Z_{2}(G')$. This makes a contradiction.
\end{proof}

\begin{lemma}\label{lem:avoidance}
Suppose that $Z$ is an extremal assignment and $G_{1}, G_{2} \in S_{2}(5)$. Assume that $v_{1}, v_{2}$ are 4-valenced vertices of $G_{1}$ and $G_{2}$ respectively. Furthermore, the labels adjacent to $v_{1}$ are $i, j, k$ and those adjacent to $v_{2}$ are $k, \ell, m$. Then at most one of $v_{1}, v_{2}$ can be assigned. 
\end{lemma}

\begin{proof}
Suppose that $K \in S_{3}(5)$ is the common degeneration of $G_{1}$ and $G_{2}$. If $v_{1} \in Z(G_{1})$ and $v_{2} \in Z(G_{2})$, then all of three vertices of $K$ must be assigned. This violates the first condition of extremal assignments.
\end{proof}

\begin{proposition}\label{prop:extass5}
There are precisely 76 equivalent classes of extremal assignments of order 5. 
\end{proposition}

\begin{proof}
By lemmas \ref{lem:assignmentandS2} and \ref{lem:avoidance}, to construct an extremal assignment, we have to choose a set of 3-sets $\cS := \{B_{i}\;|\; B_{i} \mbox { is a 3-subset of $[5]$}, B_{i} \cup B_{j} \ne [5]\}$. Then for each $\cS$, we can make an extremal assignment $Z_{\cS}$ by assigning 4-valenced vertices with 3 lables $B_{i}$ for a graph $G \in S_{2}(5)$. We will call the set $\cS$ the set of \emph{contraction indicators}. For each contraction indicator $\cS = \{B_{i}\}$, by taking the set of its complements $\{b_{i}\}$ where $b_{i} = B_{i}^{c}$, we obtain a subgraph of $K_{5}$ since $|b_{i}| = 2$. Furthermore, $B_{i} \cup B_{j}\ne [5]$ implies that $b_{i} \cap b_{j} \ne \emptyset$.Therefore, there is a one-to-one correspondence between the set of equivalent classes of extremal assignments of order 5 and 
\[
	\cC := \{H\;|\; H \mbox{ is a subgraph of } K_{5}, 
	\mbox{ any two edges of $H$ meet}\}.
\] 
There are exactly three types of such graphs: the empty graph, stars, and $K_{3}$'s. There are ${5 \choose 2} + 5 \times {4 \choose 2} + 5 \times {4 \choose 3} + 5 = 65$ stars and ${5 \choose 3} = 10$ $K_{3}$'s. 
\end{proof}

\begin{remark}
This result will be generalized in Section \ref{sec:divisorial} in terms of smooth extremal assignments. 
\end{remark}

\begin{proposition}
Every $\overline{\mathrm{M}}_{0,5}(Z)$ is isomorphic to $\overline{\mathrm{M}}_{0, A}$ for some weight data $A$. In particular, every $\overline{\mathrm{M}}_{0, 5}(Z)$ is a smooth projective variety. 
\end{proposition}

\begin{proof}
Let $k$ be the size of the associated contraction indicator. If $k = 0$, $Z$ is an empty assignment and $\overline{\mathrm{M}}_{0, 5}(Z) = \overline{\mathrm{M}}_{0, 5}$. If $k = 1$, (say the corresponding contraction indicator is $\{\{1,2,3\}\}$) then $\overline{\mathrm{M}}_{0, 5}(Z) \cong \overline{\mathrm{M}}_{0, ((1/3)^{3},1^{2})}$. If $k = 2$ and the contraction indicator is $\{\{1, 2, 3\}, \{1, 2, 4\}\}$, $\overline{\mathrm{M}}_{0, 5}(Z) \cong \overline{\mathrm{M}}_{0, ((1/4)^{2}, (1/2)^{2}, 1)}$. When $k = 3$, up to $S_{5}$-action, there are two cases. If $\{\{1, 2, 3\}, \{1, 2, 4\}, \{1, 3, 4\}\}$ is assigned, then $\overline{\mathrm{M}}_{0, 5}(Z) \cong \overline{\mathrm{M}}_{0, (1/5,(2/5)^{3},1)}$. If $\{\{1,2,3\}, \{1,2,4\}, \{1,2,5\}\}$ is assigned, then $\overline{\mathrm{M}}_{0, 5}(Z) \cong \overline{\mathrm{M}}_{0, ((1/6)^{2},(2/3)^{3})}$. Finally, if $k = 4$ and the contraction indicator is the set of 3-sets excluding 5, then $\overline{\mathrm{M}}_{0, 5}(Z) \cong \overline{\mathrm{M}}_{0, ((1/3)^{4},1)}$.
\end{proof}

\begin{remark}
However, the number of extremal assignments is already huge. There are 15 type III graphs. We may freely assign an arbitrary number of central vertices to make an extremal assignment. Therefore the number of extremal assignments of order 5 is larger than $2^{15}$. Note that all of these $2^{15}$ extremal assignments are equivalent to the empty assignment. 
\end{remark}


\section{Structure theorem for extremal assignments}\label{sec:structuretheorem}

In this section we will answer Question \ref{que:fundamentalquestion} and provide a structure theorem for extremal assignments, which will be useful to study concrete examples. 

Let $G \in S(n) \setminus S_{1}(n)$ be a star and $v$ be a central vertex of $G$. We say $(G, v)$ is called a \emph{basic pair}. For a basic pair $(G, v)$, we can assign a set partition $P := \{B_{1}, B_{2}, \cdots, B_{k}\}$ where $B_{i}$ is the set of labels on a tail of $(G, v)$. Note that $|P| \ge 3$ and there is $B_{i}$ such that $|B_{i}| \ge 2$. Conversely, for any set partition $P$ with $3 \le |P| \le n-1$, we can construct a basic pair $(G, v)$. Let $\cP(n)$ be the set of set partitions of $[n]$.

\begin{example}
Consider $G\in S_3(6)$ in Figure \ref{fig:contraction}. For the basic pair $(G, v_{1})$, the corresponding set partition is $P = \{\{1,2\},\{3\},\{4\},\{5,6\}\}$. For $(G', v')$ in Figure \ref{fig:contraction}, the corresponding set partition is $P' = \{\{1\}, \{2\}, \{3\}, \{4\}, \{5, 6\}\}$. 
\end{example}

\begin{definition}
Let $(G, v)$ be a basic pair and let $P \in \cP(n)$ be the associated set partition. 
\begin{enumerate}
\item An \emph{assignment generated by $(G, v)$} is an assignment $Z_{P}$ defined by 
\[
	Z_{P}(H) = \{w \in V(H)\;|\; \exists H \rightsquigarrow G, 
	w \rightsquigarrow v\}.
\]
\item An \emph{atomic extremal assignment} generated by $(G, v)$ is the smallest extremal assignment $Z$ such that $v \in Z(G)$. 
\end{enumerate}
\end{definition}

Let $P = \{B_{1}, B_{2}, \cdots, B_{r}\}$, $Q = \{C_{1}, C_{2}, \cdots, C_{s}\}$ be two set partitions of $[n]$. We say $P \le Q$ if for any $C_{i} \in Q$, there is $B_{j} \in P$ such that $C_{i} \subset B_{j}$. In this situation, we say that $Q$ is a \emph{refinement} of $P$ and $P$ is a \emph{corruption} of $Q$. With respect to this partial order, $\cP(n)$ is a partially ordered set. The maximum element is the complete partition $\{\{1\}, \{2\}, \cdots , \{n\}\}$, and the minimum element is $\{[n]\}$. 

\begin{definition}\label{def:vdash}
We denote by $P \preceq Q$ if $P \le Q$ and there is no $B \in P$ such that $B$ is a union of some singleton sets in $Q$. In other words, if $Q = \{\{i_{1}\}, \{i_{2}\}, \cdots, \{i_{r}\}, B_{1}, B_{2}, \cdots, B_{s}\}$ with $|B_{i}| \ge 2$, then there is no $B \in P$ such that $|B| \ge 2$ and $B \subset \{i_{1}, i_{2}, \cdots, i_{r}\}$. 
\end{definition}

\begin{lemma}\label{lem:curlyless}
Let $Q \le P$. Then there is $Q' \preceq P$ such that $Q \le Q'$. Furthermore, $Z_{Q} \subset Z_{Q'}$. 
\end{lemma}

\begin{proof}
If $B \in Q$ is a union of singleton sets in $P$, by splitting it into singleton sets, we obtain $Q'$. Let $(G_{Q}, v_{Q})$ (resp. $(G_{Q'}, v_{Q}')$) be the basic pair associated to $Q$ (resp. $Q'$). Then $G_{Q} \rightsquigarrow G_{Q'}$ and $v_{Q} \rightsquigarrow v_{Q'}$. Therefore $Z_{Q}(G_{Q})  \subset Z_{Q'}(G_{Q})$. This implies $Z_{Q}(G) \subset Z_{Q'}(G)$ for all $G$.
\end{proof}

The following proposition describes an atomic extremal assignment generated by $(G, v)$. 

\begin{proposition}\label{prop:atomic}
Let $(G, v)$ be a basic pair with a set partition $P = \{B_{1}, B_{2}, \cdots, B_{r}\}$. An atomic extremal assignment generated by $(G, v)$ is
\[
	Z := \bigcup_{Q \le P,\; |Q| \ge 3}Z_{Q}
	= \bigcup_{Q \preceq P, \;|Q| \ge 3}Z_{Q}.
\]
\end{proposition}

\begin{proof}
For $Q \le P$, let $(G_{Q}, v_{Q})$ be the associated basic pair. 

\textsf{Step 1.} The second equality holds.

If $Q \le P$, then for any $v \in Z_{Q}(G)$, by Lemma \ref{lem:curlyless}, there is $Q'$ such that $v \in Z_{Q'}(G)$, $Q \le Q'$ and $Q' \preceq P$. Thus we have
\[
	\bigcup_{Q \le P, |Q| \ge 3}Z_{Q} \subset 
	\bigcup_{Q \preceq P, \;|Q| \ge 3}Z_{Q}.
\]
The opposite inclusion is obvious.

Next we show that $Z$ is indeed an extremal assignment. 

\textsf{Step 2.} For every $H \in S(n)$, $Z(H) \ne V(H)$.

Assume that $Z(H) = V(H)$. Then by definition, $V(H) = \bigcup_{Q \preceq P, \;|Q| \ge 3}Z_{Q}(H)$, where $Z_{Q}(H) = \{v \in V(H)\;|\; \exists H \rightsquigarrow G_{Q}, v \rightsquigarrow v_{Q}\}$. For each vertex $w_{i}$ adjacent to label $i$, $w_{i} \in Z_{Q}(H)$ for some $Q \le P$. Then $Q$ has a singleton set $\{i\}$. Since $P$ is a refinement of $Q$, $\{i\} \in P$. Therefore $P$ is the complete partition, which is a contradiction. Therefore $Z(H) \ne V(H)$.

\textsf{Step 3.} If $\pi : H \rightsquigarrow H'$ and $\{v_{1}, v_{2}, \cdots, v_{k}\} \rightsquigarrow v'$, then $v' \in Z(H')$ if and only if $\{v_{1}, v_{2}, \cdots, v_{k}\} \in Z(H)$.

Suppose that $v' \in Z(H')$. By definition of $Z$, there is $Q \preceq P$ such that $H' \rightsquigarrow G_{Q}$, $v' \rightsquigarrow v_{Q}$. Then $H \rightsquigarrow G_{Q}$ and $\{v_{1}, v_{2}, \cdots, v_{k}\} \rightsquigarrow v_{Q}$. Thus $v_{1}, v_{2}, \cdots, v_{k} \in Z(H)$. 

Conversely, assume that $\{v_{1}, v_{2}, \cdots, v_{k}\} \in Z(H)$. We will use induction on $k$. Consider the case of $k = 1$. There is $Q \preceq P$ with corresponding basic pair $(G_{Q}, v_{Q})$ so that $q : H \rightsquigarrow G_{Q}$ and $v \rightsquigarrow v_{Q}$. On the other hand, let $(G_{Q'}, v_{Q'})$ be the basic pair obtained from $H'$ by 1) contracted all non-contracted vertices in $\pi(q^{-1}(v_{Q}))$ (including $v'$) to a single vertex $v_{Q'}$ and 2) contract each tail connected to $v_{Q'}$ to a vertex. Let $Q'$ be the associated set partition. Then $Q' \le Q$. Thus $Q' \le P$ and $v' \in H'$ is in $Z(H')$. See Figure \ref{fig:k=1}.

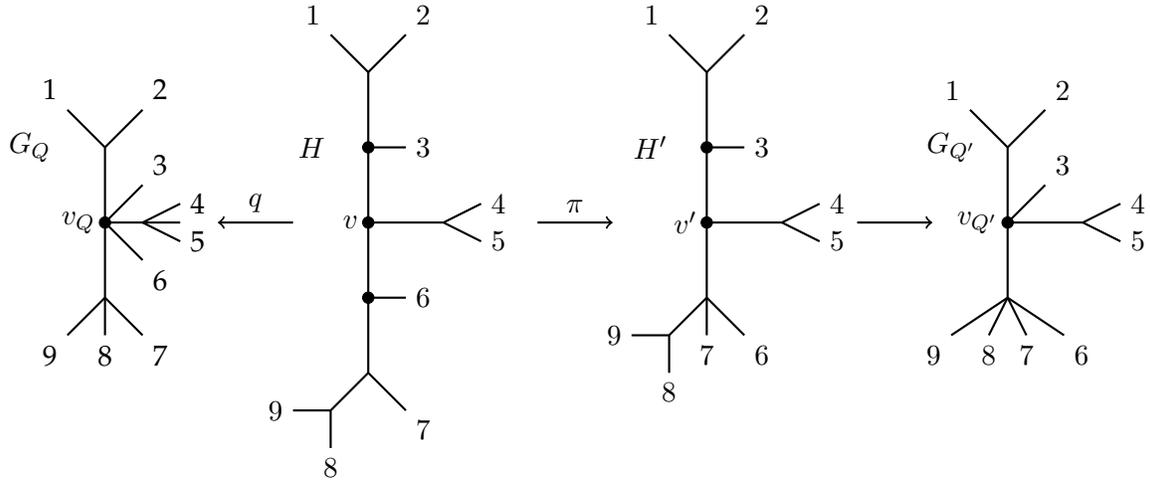
\begin{figure}[!ht]
\begin{tikzpicture}[scale=0.5]
\draw [thick] (0,-4) -- (0,4);
\draw [thick] (0,0) -- (2,0);
\draw [thick] (0,2) -- (1,2);
\node [right] at (1,2) {$3$};
\draw [thick] (0,-2) -- (1,-2);
\node [right] at (1,-2) {$6$};
\draw [thick] (0,4) -- (-1,5);
\node [above left] at (-1,5) {$1$};
\draw [thick] (0,4) -- (1,5);
\node [above right] at (1,5) {$2$};
\draw [thick] (0,-4) -- (1,-5);
\node [below right] at (1,-5) {$7$};
\draw [thick] (0,-4) -- (-1,-5);
\draw [thick] (-1,-5) -- (-2,-5);
\node [left] at (-2,-5) {$9$};
\draw [thick] (-1,-5) -- (-1,-6);
\node [below] at (-1,-6) {$8$};
\draw[thick] (2,0) -- (3,.5);
\node [right] at (3,.5) {$4$};
\draw[thick] (2,0) -- (3,-.5);
\node [right] at (3,-.5) {$5$};

\draw [thick] (9,-3) -- (9,4);
\node [below] at (9,-3) {$7$};
\draw [thick] (9,0) -- (11,0);
\draw [thick] (9,2) -- (10,2);
\node [right] at (10,2) {$3$};
\draw [thick] (9,4) -- (8,5);
\node [above left] at (8,5) {$1$};
\draw [thick] (9,4) -- (10,5);
\node [above right] at (10,5) {$2$};
\draw [thick] (9,-2) -- (10,-3);
\node [below right] at (10,-3) {$6$};
\draw [thick] (9,-2) -- (8,-3);
\draw [thick] (8,-3) -- (7,-3);
\node [left] at (7,-3) {$9$};
\draw [thick] (8,-3) -- (8,-4);
\node [below] at (8,-4) {$8$};
\draw [thick] (11,0) -- (12,.5);
\node [right] at (12,.5) {$4$};
\draw [thick] (11,0) -- (12,-.5);
\node [right] at (12,-.5) {$5$};

\draw [thick] [->] (4.5,0) -- (6.5,0);
\node [above] at (5.5,0) {$\pi$};

\draw [thick] (-7,-3) -- (-7,2);
\node [below] at (-7,-3) {8};
\draw [thick] (-7,0) -- (-5,0);
\draw [thick] (-7,0) -- (-6,1);
\node [above right] at (-6,1) {3};
\draw [thick] (-7,0) -- (-6,-1);
\node [below right] at (-6,-1) {6};
\draw [thick] (-7,2) -- (-8,3);
\node [above left] at (-8,3) {1};
\draw [thick] (-7,2) -- (-6,3);
\node [above right] at (-6,3) {2};
\draw [thick] (-7,-2) -- (-8,-3);
\node [below left] at (-8,-3) {9};
\draw [thick] (-7,-2) -- (-6,-3);
\node [below right] at (-6,-3) {7};
\draw [thick] (-6,0) -- (-5,.5);
\node [right] at (-5,.5) {4};
\draw [thick] (-6,0) -- (-5,-0.5);
\node [right] at (-5,-0.5) {5};
\draw [fill] (-7,0) circle [radius=0.15];
\node [left] at (-7,0) {${v}_{Q}$};
\node at (-9,2) {${G}_{Q}$};

\draw [thick] [->] (-2, 0) -- (-4, 0);
\node [above] at (-3,0) {$q$};

\draw [thick] (17,-2) -- (17,2);
\draw [thick] (17,0) -- (19,0);
\draw [thick] (19,0) -- (20,.5);
\node [right] at (20,.5) {$4$};
\draw [thick] (19,0) -- (20,-0.5);
\node [right] at (20,-0.5) {$5$};
\draw [thick] (17,0) -- (18,1);
\node [above right] at (18,1) {$3$};
\draw [thick] (17,2) -- (16,3);
\node [above left] at (16,3) {$1$};
\draw [thick] (17,2) -- (18,3);
\node [above right] at (18,3) {$2$};
\draw [thick] (17,-2) -- (15.5,-3);
\node [below left] at (15.5,-3) {$9$};
\draw [thick] (17,-2) -- (16.5,-3);
\node [below] at (16.5,-3) {$8$};
\draw [thick] (17,-2) -- (17.5,-3);
\node [below] at (17.5,-3) {$7$};
\draw [thick] (17,-2) -- (18.5,-3);
\node [below right] at (18.5,-3) {$6$};
\draw [fill] (17,0) circle [radius=0.15];
\node [left] at (17,0) {${v}_{Q'}$};
\node at (15.5,2) {${G}_{Q'}$};

\draw [thick] [->] (13,0) -- (15,0);

\draw [fill] (0,0) circle [radius=0.15];
\node [left] at (0,0) {$v$};
\draw [fill] (0,2) circle [radius=0.15];
\draw [fill] (0,-2) circle [radius=0.15];
\draw [fill] (9,0) circle [radius=0.15];
\node [left] at (9,0) {$v'$};
\draw [fill] (9,2) circle [radius=0.15];
\node at (-1.5,2) {$H$};
\node at (7.5,2) {$H'$};
\end{tikzpicture}
\caption{$k=1$ case}
\label{fig:k=1}
\end{figure}

Now consider a general case. Since $\{v_{1}, v_{2}, \cdots, v_{k}\}$ forms a connected subgraph of $H$, there are two adjacent vertices (say $v_{1}, v_{2}$). Then by definition of $Z$, $v_{1} \in Z_{Q}(H)$ and $v_{2} \in Z_{R}(H)$ for some $Q, R \preceq P$. If $Q = R$, then we can contract $v_{1}$ and $v_{2}$ to a single vertex $w$ and make a new graph $\overline{H}$. Then still $\overline{H} \rightsquigarrow G_{Q}$ and $w \rightsquigarrow v_{Q}$, thus $w \in Z(\overline{H})$. Now $\overline{H} \rightsquigarrow H'$ and $\{w, v_{3}, v_{4}, \cdots, v_{k}\} \rightsquigarrow v'$. So it is reduced to the $k-1$ vertices case, which is true by the induction hypothesis. If $Z_{Q}(H) \cap Z_{R}(H) \ne \emptyset$, then both $v_{1}, v_{2}$ are in one of them, say, $Z_{Q}(H)$. Then if we replace $R$ by $Q$, it is reduced to the $Q = R$ case. 

Finally, suppose that $Z_{Q}(H) \cap Z_{R}(H) = \emptyset$. If $Q = \{C_{1}, C_{2}, \cdots, C_{s}\}$ and $R = \{D_{1}, D_{2}, \cdots, D_{t}\}$, then there are $C_{i}$ and $D_{j}$ such that $C_{i} = \bigcup_{k \ne j}D_{k}$ and $D_{j} = \bigcup_{k \ne i}C_{k}$. Since $Q$ and $R$ are corruptions of $P$, $P$ is a refinement of $P' := \{C_{k}\}_{k \ne i} \bigcup \{D_{k}\}_{k \ne j}$. Then $P' \preceq P$ and for $H \rightsquigarrow G_{P'}$, both $v_{1}, v_{2}$ are contracted to the central vertex $v_{P'} \in V(G_{P'})$. Therefore, after replacing $Q, R$ by $P'$, we are reduced to the case of $Q = R$. 

\begin{figure}[!hb]
\begin{tikzpicture}[scale=0.7]
\draw [thick] (0,0) -- (0,3);
\node [above] at (0,3) {2};
\draw [thick] (0,2) -- (-1,3);
\node [above left] at (-1,3) {$1$};
\draw [thick] (0,2) -- (1,3);
\node [above right] at (1,3) {$3$};
\draw [thick] (-1,0) -- (1,0);
\node [left] at (-1,0) {$9$};
\node [right] at (1,0) {$4$};
\draw [thick] (0,0) -- (-1,-1);
\node [below left] at (-1,-1) {$8$};
\draw [thick] (0,0) -- (1,-1);
\node [below right] at (1,-1) {$5$};
\draw [thick] (0,0) -- (-.33,-1);
\node [below] at (-.33,-1) {$7$};
\draw [thick] (0,0) -- (.33,-1);
\node [below] at (.33,-1) {$6$};

\draw [thick] (7,-3) -- (7,3);
\node [above] at (7,3) {$2$};
\node [below] at (7,-3) {$8$};
\draw [thick] (7,0) -- (9,0);
\node [right] at (9,0) {$5$};
\draw [thick] (7,0) -- (9,1);
\node [above right] at (9,1) {$4$};
\draw [thick] (7,0) -- (9,-1);
\node [below right] at (9,-1) {$6$};
\draw [thick] (7,2) -- (6,3);
\node [above left] at (6,3) {$1$};
\draw [thick] (7,2) -- (8,3);
\node [above right] at (8,3) {$3$};
\draw [thick] (7,-2) -- (6,-3);
\node [below left] at (6,-3) {$9$};
\draw [thick] (7,-2) -- (8,-3);
\node [below right] at (8,-3) {$7$};

\draw [thick] (14,2) -- (14,-1);
\node [below] at (14,-1) {8};
\draw [thick] (14,0) -- (13,-1);
\node [below left] at (13,-1) {9};
\draw [thick] (14,0) -- (15,-1);
\node [below right] at (15,-1) {7};
\draw [thick] (13,2) -- (15,2);
\node [left] at (13,2) {1};
\node [right] at (15,2) {6};
\draw [thick] (14,2) -- (13,3);
\node [above left] at (13,3) {2};
\draw [thick] (14,2) -- (15,3);
\node [above right] at (15,3) {5};
\draw [thick] (14,2) -- (13.66,3);
\node [above] at (13.66,3) {3};
\draw [thick] (14,2) -- (14.33,3);
\node [above] at (14.33,3) {4};
\draw [thick] (14,2) -- (13,3);
\draw [fill] (0,2) circle [radius=0.1];
\draw [fill] (7,0) circle [radius=0.1];
\draw [fill] (14,2) circle [radius=0.1];
\node at (0,-3) {$Q = \{\{1\},\{2\},\{3\},\{4,5,6,7,8,9\}\}$};
\node at (7,4.5) {$R = \{\{1,2,3\},\{4\},\{5\},\{6\},\{7,8,9\}\}$};
\node at (14,-3) {$P' = \{\{1\},\{2\},\{3\},\{4\},\{5\},\{6\},\{7,8,9\}\}$};
\end{tikzpicture}
\caption{The case that $Z_{Q}(H) \cap Z_{R}(H) = \emptyset$ in Step 3.}
\label{fig:disjoint}
\end{figure}
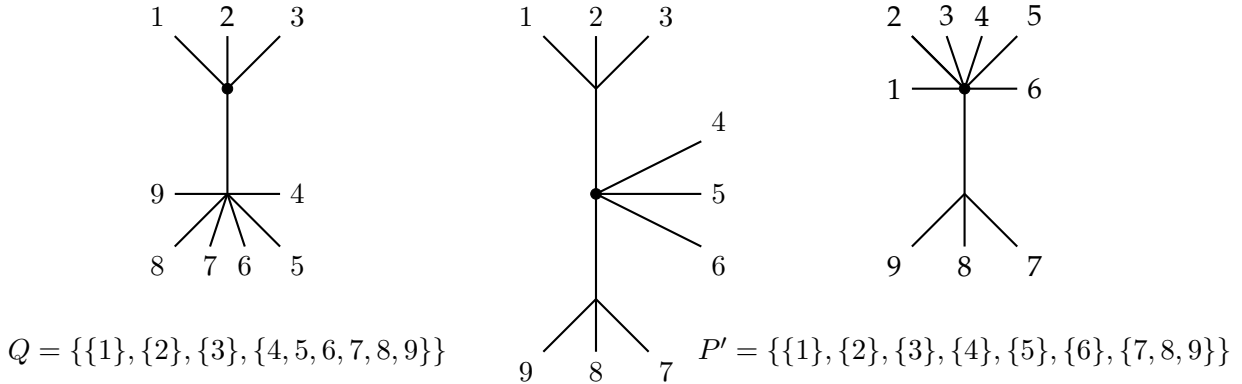

\textsf{Step 4.} $Z$ is the smallest extremal assignment that $v \in Z(G)$. 

Note that for any extremal assignment $Z'$ with $v \in Z'(G)$, by applying degeneration/contraction operations, one can check that $v_{Q} \in Z'(G_{Q})$ for $Q \preceq P$. By the second axiom of extremal assignments, the assignment $Z_{Q}$ generated by $(G_{Q}, v_{Q})$ is contained in $Z'$. Therefore $Z = \bigcup_{Q \preceq P, |Q| \ge 3}Z_{Q}$ is in $Z'$.
\end{proof}

\begin{remark}\label{rem:connectivity}
Let $Z$ be an atomic extremal assignment. Then for any $G \in S(n)$, $Z(G)$ is a connected subset of $V(G)$ if it is nonempty. 
\end{remark}

Every extremal assignment can be described using atomic extremal assignments. 

\begin{lemma}\label{lem:extremalssignmentunionofatomic}
Any nonempty extremal assignment $Z$ is a union of finitely many atomic extremal assignments. 
\end{lemma}

\begin{proof}
For any $G \in S(n)$ and $v \in Z(G)$, take the maximal connected subset $W$ of $Z(G)$ containing $v$. By contracting $W$ to a single vertex $v'$, we obtain $G'$ and $v' \in Z(G')$. By contracting each tail to a vertex adjacent to $v'$, we obtain a star $G''$ and a central vertex $v''$ so that $G \rightsquigarrow G''$, $v \rightsquigarrow v''$, and $v'' \in Z(G'')$. Let $Z_{v}$ be the atomic extremal assignment generated by $(G'', v'')$. Then $v \in Z_{v}(G)$ and $Z_{v} \subset Z$, since it is the smallest extremal assignment where $v$ is assigned. Now it is clear that 
\[
	Z = \bigcup_{v \in Z(G), G \in S(n)}Z_{v}.
\]
\end{proof}

\begin{definition}\label{def:commonupperbound}
\begin{enumerate}
\item Two partitions $P, Q \in \cP(n)$ are called \emph{incompatible} if $P \not\le Q$ and $Q \not\le P$. 
\item For $P, Q \in \cP(n)$, we say that $R$ is a \emph{common upper bound} if $P \le R$ and $Q \le R$. A common upper bound $R$ is \emph{tight} if every $B \in R$ is in either $P$ or $Q$. 
\item Two partitions $P \ne Q \in \cP(n)$ are \emph{transversal} if $P, Q$ do not have a tight common upper bound. $P, Q$ are \emph{strongly transversal} if for every $P' \preceq P$, $Q' \preceq Q$ such that $|P'|, |Q'| \ge 3$ and $P'$ and $Q'$ are incompatible, $P'$ and $Q'$ are transversal. 
\end{enumerate}
\end{definition}

\begin{proposition}\label{prop:extremaltightbound}
Let $Z$ be an extremal assignment. 
\begin{enumerate}
\item There are finitely many atomic extremal assignments $Z_{1}, Z_{2}, \cdots, Z_{k}$ with corresponding set partitions $P_{1}, P_{2}, \cdots, P_{k}$ such that 
\begin{enumerate}
\item $P_{i}$ is not the complete partition;
\item If $i \ne j$, $P_{i}$ and $P_{j}$ are incompatible;
\item $Z = \bigcup_{i=1}^{k}Z_{i}$.
\end{enumerate}
\item Suppose that $Q_{i} \preceq P_{i}$ and $Q_{j} \preceq P_{j}$. If $R$ is a tight common upper bound of $Q_{i}$ and $Q_{j}$, then $R \le P_{\ell}$ for some $P_{\ell}$.
\end{enumerate}
\end{proposition}

\begin{proof}
The existence of finitely many atomic extremal assignments is obtained by Lemma \ref{lem:extremalssignmentunionofatomic}. By Proposition \ref{prop:atomic}, if $P_{i} \le P_{j}$, then $Z_{i} \subset Z_{j}$. Therefore, by eliminating non-maximal $P_{i}$'s, we obtain a finite collection of set partitions satisfying (a), (b), and (c). This proves (1).

Let $(G, v)$ be the basic pair associated to $R$. Then there is a degeneration $F \rightsquigarrow G$ such that 
\begin{enumerate}
\item $\{v_{1}, v_{2}\} \rightsquigarrow v$;
\item Two basic pairs $(F_{i}, w_{i})$ and $(F_{j}, w_{j})$ associated to $Q_{i}$ and $Q_{j}$ respectively have the following property: there is a contraction $F \rightsquigarrow F_{i}$ (resp. $F \rightsquigarrow F_{j}$) so that $v_{1} \rightsquigarrow w_{i}$ (resp. $v_{2} \rightsquigarrow w_{j}$). 
\end{enumerate}
Since $w_{i} \in Z(F_{i})$ and $w_{j} \in Z(F_{j})$, $v_{1}, v_{2} \in Z(F)$. Then $v \in Z(G)$. Therefore $v \in Z_{\ell}(G)$ for some $\ell$. Then by Proposition \ref{prop:atomic}, $R \le P_{\ell}$. 
\end{proof}

The following is a converse. 

\begin{theorem}\label{thm:structurethm}
Let $\{P_{1}, P_{2}, \cdots, P_{k}\}$ be a set of incomplete set partitions of order $n$ such that $P_{i} \not\le P_{j}$ for any $i \ne j$. Let $(G_{i}, v_{i})$ be the basic pair corresponding to $P_{i}$ and let $Z_{i}$ be the atomic extremal assignment generated by $(G_{i}, v_{i})$. Suppose that for any $i \ne j$ and $Q_{i} \preceq P_{i}$ and $Q_{j} \preceq P_{j}$, if $R$ is a tight common upper bound of $Q_{i}$ and $Q_{j}$, then there is $P_{\ell}$ such that $R \le P_{\ell}$. Then $Z := \bigcup_{i=1}^{k}Z_{i}$ is extremal.
\end{theorem}

\begin{proof}
\textsf{Step 1.} For every $H \in S(n)$, $Z(H) \ne V(H)$.

Assume not. So there exists $H$ with $V(H) = Z(H) = \bigcup_{i=1}^{k}Z_{i}(H)$. Suppose that for $i \ne j$, $Z_{i}(H) \cap Z_{j}(H) \ne \emptyset$. Let $W = Z_{i}(H) \cap Z_{j}(H)$. Both $Z_{i}(H)$ and $Z_{j}(H)$ are connected subsets (Remark \ref{rem:connectivity}) of a tree; $W$ is connected too. 
If $Z_{i}(H) \ne Z_{j}(H)$, then we may assume that there is $v' \in Z_{i}(H)\setminus Z_{j}(H)$ that is adjacent to $v \in W$. Let $H \rightsquigarrow \overline{H}$ be the contraction such that $\{v, v'\} \rightsquigarrow w$. Since $Z_{i}, Z_{j}$ are extremal, $w \in Z_{i}(\overline{H})$ but $w \notin Z_{j}(\overline{H})$. But still $Z(\overline{H}) = V(\overline{H})$. By applying this procedure several times, we may assume that for two $i$ and $j$, either $Z_{i}(H) = Z_{j}(H)$ or $Z_{i}(H) \cap Z_{j}(H) = \emptyset$. 

Next, take a contraction $H \rightsquigarrow \overline{H}$ which contracts $Z_{i}(H)$ to a single vertex $w_{i}$. Then still $w_{i} \in Z_{i}(\overline{H})$ and $Z_{j}(\overline{H})$ can be identified with $Z_{j}(H)$. Thus $Z(\overline{H}) = V(\overline{H})$. Thus we may assume that each $Z_{i}(H)$ is a single vertex. Let $Z_{i}(H) = \{w_{i}\}$. 

By contracting each tail connected to $w_{i}$ into a single vertex, we can obtain a basic pair $(F_{i}, w_{i})$ whose corresponding set partition is $Q_{i} = \{C_{1}, C_{2}, \cdots, C_{s}\}$. By Proposition \ref{prop:atomic}, $Q_{i} \preceq P_{i}$. Similarly, we can construct a set partition $Q_{j} = \{D_{1}, D_{2}, \cdots, D_{t}\} \preceq P_{j}$. Then there are $C_{i}$ and $D_{j}$ such that $C_{i} = \bigcup_{k \ne j}D_{k}$ and $D_{j} = \bigcup_{k \ne i}C_{k}$. So $Q_{i}$ and $Q_{j}$ are incompatible, but $R := \{C_{k}\}_{k \ne i}\cup \{D_{k}\}_{k \ne j}$ is a tight common upper bound of them. By the assumption, there is $P_{\ell}$ such that $R \le P_{\ell}$. 

On the other hand, if we contract $w_{i}$ and $w_{j}$ to a vertex $x$ and contract each tail to a vertex, we obtain a basic pair $(F, x)$, which corresponds to $R$. Thus $x \in Z_{\ell}(F)$. Since $H \rightsquigarrow F$ and $\{w_{i}, w_{j}\} \rightsquigarrow x$, we have $w_{i}, w_{j} \in Z_{\ell}(H)$. This makes a contradiction, since $Z_{\ell}(H)$ is a singleton set. Therefore $Z(H) \ne V(H)$. 

\textsf{Step 2.} If $H \rightsquigarrow H'$ and $\{v_{1}, v_{2}, \cdots, v_{k}\} \rightsquigarrow v'$, then $v' \in Z(H')$ if and only if $v_{1}, v_{2}, \cdots, v_{k} \in Z(H)$. 

If $v' \in Z(H')$, then $v' \in Z_{i}(H')$ for some $i$. Then $v_{1}, v_{2}, \cdots, v_{k} \in Z_{i}(H) \subset Z(H)$. 

Conversely, suppose that $v_{1}, v_{2}, \cdots, v_{k} \in Z(H)$. Take two adjacent vertices $v_{1}, v_{2} \in Z(H)$. Then $v_{1} \in Z_{i}(H)$ and $v_{2} \in Z_{j}(H)$ for some $i$ and $j$. If $i = j$, then contract $v_{1}$ and $v_{2}$ to a single vertex $w$ and make a new graph $\overline{H}$. Then $\overline{H} \rightsquigarrow H'$, $w \rightsquigarrow v'$, and $w \in Z_{i}(\overline{H}) \subset Z(\overline{H})$. Therefore we can reduce to the $(k-1)$ vertices case. If $Z_{i}(H) \cap Z_{j}(H) \ne \emptyset$, one of $Z_{i}(H)$ and $Z_{j}(H)$ (say $Z_{i}(H)$) contains both $v_{1}$ and $v_{2}$, so by replacing $Z_{j}(H)$ by $Z_{i}(H)$, we can reduce to the $i = j$ case. 

Suppose that $Z_{i}(H) \cap Z_{j}(H) = \emptyset$. Then as in Step 1, there is a basic pair $(F_{i}, w_{i})$ (resp. $(F_{j}, w_{j})$) with corresponding set partition $Q_{i} \preceq P_{i}$ (resp. $Q_{j} \preceq P_{j}$) with a tight common upper bound $R \le P_{\ell}$. Then $w_{i} \in Z_{\ell}(F_{i})$ and $w_{j} \in Z_{\ell}(F_{j})$. So $Z_{i}(H), Z_{j}(H) \subset Z_{\ell}(H)$. By replacing $Z_{i}$, $Z_{j}$ by $Z_{\ell}$, we can reduce to the $i = j$ case.

So we deduce to the case of $k = 1$. In this case $v_{1} \in Z_{i}(H)$ implies $v' \in Z_{i}(H') \subset Z(H')$, since $Z_{i}$ is extremal.
\end{proof}

An immediate simple consequence is:

\begin{corollary}\label{cor:unionextremal}
Let $\{P_{1}, P_{2}, \cdots, P_{k}\}$ be a set of incomplete set partitions of order $n$ such that $P_{i} \not\le P_{j}$ for any $i \ne j$. Let $(G_{i}, v_{i})$ be the basic pair associated to $P_{i}$ and let $Z_{i}$ be the atomic extremal assignment generated by $(G_{i}, v_{i})$. If $P_{i}$ and $P_{j}$ are strongly transversal for any $i \ne j$, then $Z := \bigcup_{i=1}^{k}Z_{i}$ is extremal.
\end{corollary}

\begin{proof}
If $Q_{i} \preceq P_{i}$ and $Q_{j} \preceq P_{j}$ are incompatible, then by assumption there is no tight common upper bound. If $Q_{i} \le Q_{j}$, then their tight common upper bound is $Q_{j}$. 
\end{proof}

\begin{example}
Let $P_1 = \{\{1,2\},\{3,4\},\{5,6,7,8\}\}$ correspond to a basic pair $(G_1,v_1)$ and $P_2 = \{\{1,2,3,4\},\{5,6\},\{7,8\}\}$ correspond to $(G_2,v_2)$. Let $Z_{i}$ be an atomic extremal assignment associated to $(G_{i}, v_{i})$. Then there is a tight upper bound $P_3 = \{\{1,2\},\{3,4\},\{5,6\},\{7,8\}\}$. We can degenerate the vertex in $P_2$ with labels $\{5,6,7,8\}$ into two vertices labelled $\{5,6\}$ and $\{7,8\}$ to obtain a new graph $H$, with one of two central vertices assigned. However, we can degenerate the vertex in $P_1$ with labels $\{1,2,3,4\}$ into two vertices labelled $\{1,2\}$ and $\{3,4\}$ to obtain $H$ with only the other central vertex assigned. Thus, in the union of these extremal assignments both of these adjacent vertices are assigned. Then we can contract these vertices to obtain a star graph $H'$ with central vertex $w$ and corresponding partition $P_3$. If the union is extremal, $w \in Z_{1} \cup Z_{2}(H')$. But $w\not\in Z_1(H')$ and $w\not\in Z_2(H')$, so $Z_1\bigcup Z_2$ is not an extremal assignment.
\end{example}

Due to Theorem \ref{thm:structurethm}, we can approach Question \ref{que:fundamentalquestion} for extremal assignments algorithmically.

\begin{algorithm}[Existence/construction of the smallest extremal assignment]
Let $G_{1}, G_{2}, \cdots, G_{k} \in S(n)$ and $v_{1} \in V(G_{1}), v_{2} \in V(G_{2}), \cdots, v_{k} \in V(G_{k})$. We want to find the smallest extremal assignment $Z$ such that $v_{i} \in Z(G_{i})$, if there is one.
\begin{enumerate}
\item Contract each tail of $G_{i}$ adjacent to $v_{i}$ to a single vertex and make a basic pair $(\overline{G}_{i}, v_{i})$. Let $P_{i}$ be the corresponding set partition.
\item For the family $\cF := \{P_{1}, P_{2}, \cdots, P_{k}\}$, take $Q_{i} \preceq P_{i}$ and $Q_{j} \preceq P_{j}$ and compute the tight upper bound $R$, if it exists. Add $R$ to $\cF$. 
\item Repeat step (2) until all tight upper bounds are in $\cF$. 
\item If $\cF$ contains the complete set partition $\{\{1\}, \{2\}, \cdots, \{n\}\}$, such $Z$ does not exist. 
\item If $\cF$ does not contain the complete set partition, then eliminate non-maximal elements in $\cF$. 
\item Let $Z_{P}$ be the atomic extremal assignment associated to $P$. Then $Z = \bigcup_{P \in \cF}Z_{P}$ is the smallest extremal assignment. 
\end{enumerate}
\end{algorithm}

The algorithm is implemented as a program in Sage. It can be found on the website of the first author:
\begin{center}
	\url{http://www.hanbommoon.net/publications/extremal}
\end{center}


\section{Smooth extremal assignments}\label{sec:divisorial}

In this section, we investigate a special class of extremal assignments that provides smooth contractions of $\Mznb$. 

\subsection{Definition and basic properties}

\begin{definition}
An extremal assignment $Z$ is called \emph{smooth} if for any $G \in S(n)$ and $v \in Z(G)$, there is $G' \in S_{2}(n)$ and $v' \in Z(G')$ such that $G \rightsquigarrow G'$, $v \rightsquigarrow v'$. 
\end{definition}

An assignment on $S_{2}(n)$ completely determines a smooth extremal assignment. 

\begin{lemma}\label{lem:divisorialdeterminedbyS2}
Let $Z_{1}, Z_{2}$ be two smooth extremal assignments such that $Z_{1}(G) = Z_{2}(G)$ for all $G \in S_{2}(n)$. Then $Z_{1} = Z_{2}$. 
\end{lemma}

\begin{proof}
If $v \in Z_{1}(H)$ for some $H \in S(n)$, then there is $G \in S_{2}(n)$, $w \in Z_{1}(G)$ such that $H \rightsquigarrow G$ and $v \rightsquigarrow w$. Then $w \in Z_{2}(G)$, and from the second axiom of extremal assignments, $v \in Z_{2}(H)$. Therefore $Z_{1} \subset Z_{2}$. By symmetry, $Z_{2} \subset Z_{1}$. 
\end{proof}

An important family of smooth extremal assignments is given by weight assignments. 

\begin{lemma}\label{lem:Hassettisdivisorial}
Let $Z_{A}$ be a weight assignment (Section \ref{ssec:weightassignment}) associated to weight data $A$. Then $Z_{A}$ is smooth.
\end{lemma}

\begin{proof}
Let $G \in S(n)$ and $v \in Z_{A}(G)$. Then by Definition \ref{def:Hassettextremalassignment}, there is a tail $T \subset G$ such that $v \in V(T)$. Note that any $v' \in V(T)$ is assigned too. Let $G' \in S_{2}(n)$ be a contraction of $G$ obtained by contracting $T$ to a single vertex $w$ and contracting $T^{c}$ to another vertex $w'$. By the second axiom of extremal assignments, $w \in Z_{A}(G')$. Therefore $Z_{A}$ is smooth. 
\end{proof}

However, there are smooth extremal assignments that are not weight assignments. 

\begin{example}\label{ex:divisorialbutnotHassett}
Suppose that $n = 6$. We assign two vertices of graphs in $S_{2}(6)$, whose label sets are $\{1, 2, 3\}$ and $\{2, 5, 6\}$. Then $Z = Z_{P_{1}} \cup Z_{P_{2}}$ is the union of two atomic extremal assignments generated by two partitions $P_{1} := \{\{1\}, \{2\}, \{3\}, \{4, 5, 6\}\}$ and $P_{2} := \{\{2\}, \{5\}, \{6\}, \{1, 3, 4\}\}$. It is straightforward to see that $P_{1}$ and $P_{2}$ are strongly transversal, so by Corollary \ref{cor:unionextremal}, $Z$ is indeed an extremal assignment. If $Z = Z_{A}$ for some weight data $A = (a_{1}, a_{2}, \cdots, a_{6})$, then we have two inequalities:
\[
	a_{1}+a_{2}+a_{3} \le 1, \; 
	a_{2}+a_{5}+a_{6} \le 1.
\]
We have
\[
	\frac{1}{2}(a_{1}+a_{3}) \le \frac{1}{2}(1-a_{2}), \;
	\frac{1}{2}(a_{5}+a_{6}) \le \frac{1}{2}(1-a_{2}).
\]
The left hand side is an average of two weights, so we know that at least one of $a_{1}$ and $a_{3}$ (resp. $a_{5}$ and $a_{6}$) must be less than the right hand side. Suppose that $a_{1}, a_{5} \le \frac{1}{2}(1-a_{2})$. Then 
\[
	a_{1}+a_{5} \le 1-a_{2} \Rightarrow a_{1}+a_{2}+a_{5}\le 1.
\]
This implies that the vertex with label set $\{1, 2, 5\}$ must be assigned too. 
\end{example}

Lemma \ref{lem:divisorialdeterminedbyS2} tells us that smooth extremal assignments can be described by simpler combinatorial data. Note that any graph $G \in S_{2}(n)$ is a star and any vertex $v \in V(G)$ can be a central vertex. If we take the set of labels adjacent to $v$, the collection of these sets determines the smooth extremal assignment. 

\begin{definition}\label{def:contractionindicator}
A collection $\cC$ of subsets of $[n]$ is called a \emph{contraction indicator} if it satisfies
\begin{enumerate}
\item For any $B \in \cC$, $2 \le |B| \le n-2$;
\item If $B \in \cC$ and $B' \subset B$ with $|B'| \ge 2$, then $B' \in \cC$;
\item For $B_{1}, B_{2} \in \cC$, $B_{1} \cup B_{2} \ne [n]$.
\end{enumerate}
\end{definition}

The following lemma is straightforward. 

\begin{lemma}\label{lem:assignmenttocontractionindicator}
Let $Z$ be an extremal assignment. For $G \in S_{2}(n)$ such that $Z(G) \ne \emptyset$, let $B_{G}$ be the set of all labels adjacent to the assigned vertex. Let $\cC = \{B_{G}\;|\; G \in S_{2}(n)\}$. Then $\cC$ is a contraction indicator.
\end{lemma}

The converse is also true.

\begin{proposition}\label{prop:divassignment}
Let $\cC$ be a contraction indicator. Define an assignment $Z$ for $S(n)$ as 
\begin{equation}\label{eqn:assignmentfromcontractionindicator}
	Z(G) = \{v \in V(G)\;|\;\exists \;H \in S_{2}(n), w \in V(H), 
	\exists \;\ell(w) \in \cC, 
	G \rightsquigarrow H, v \rightsquigarrow w\}.
\end{equation}
Then $Z$ is a smooth extremal assignment.
\end{proposition}

\begin{proof}
Let $\cC^{M} := \{B_{1}, B_{2}, \cdots, B_{k}\}$ be the set of maximal elements in $\cC$. Let $P_{i} = \{\{j\}\}_{j \in B_{i}} \cup \{B_{i}^{c}\}$ be the set partition consisting of singleton sets $\{j\}$ if $j \in B_{i}$ and $B_{i}^{c}$. Finally, let $Z_{i}$ be the atomic extremal assignment corresponding to $P_{i}$. 

First of all, we claim that $Z = \bigcup_{i=1}^{k}Z_{i}$. For any $v \in Z(G)$, from the definition of $Z$, there is $Q \in \cP(n)$ such that $v \in Z_{Q}(G)$. From the maximality, we can find $P_{i}$ such that $Q \preceq P_{i}$. Therefore $v \in Z_{i}(G)$. This implies $Z(G) \subset \bigcup_{i=1}^{k}Z_{i}(G)$. By definition of $Z$ again, the opposite inclusion is obvious. 

Now we will show that $P_{1}, P_{2}, \cdots, P_{k}$ are pairwise strongly transversal. Let $Q_{i} \preceq P_{i}$, $Q_{j} \preceq P_{j}$ and $Q_{i}$ and $Q_{j}$ be incompatible. Since $P_{i}$ has only one non-singleton set $B_{i}^{c}$, $Q_{i}$ has only one non-singleton set $C_{i} \supset B_{i}^{c}$ too from Definition \ref{def:vdash}. Similarly, $P_{j}$ has $B_{j}^{c}$ and $Q_{j}$ has $C_{j} \supset B_{j}^{c}$. Because $C_{i}^{c}, C_{j}^{c} \in C$, we have $C_{i}^{c} \cup C_{j}^{c} \ne [n]$. Thus $C_{i} \cap C_{j} \ne \emptyset$. If $C_{i} \subset C_{j}$, then $Q_{i} \le Q_{j}$. Therefore $C_{i} \not\subset C_{j}$ and $C_{j} \not\subset C_{i}$. Therefore $Q_{i}$ and $Q_{j}$ do not have a tight common refinement. 

By Corollary \ref{cor:unionextremal}, the union $Z$ is extremal.
\end{proof}

Smooth assignments are special among extremal assignments because of the following result. 

\begin{theorem}\label{thm:divisorialsmooth}
The birational model $\Mznb(Z)$ is smooth if and only if $Z$ is equivalent to a smooth extremal assignment. 
\end{theorem}

\begin{proof}
Suppose that $Z$ is smooth. Any $(C, x_{1}, \cdots, x_{n}) \in \Mznb(Z)$ is obtained from $(C^{s}, x_{1}^{s}, \cdots, y_{n}^{s}) \in \Mznb$ by contracting some (not necessarily irreducible) tails. In particular, $C$ is a nodal curve and $x_{i}$'s are not necessarily distinct smooth points on $C$. Since the stack of nodal curves with smooth marked points is smooth, $\Mznb(Z)$ is smooth. 

Conversely, suppose that $\Mznb(Z)$ is a smooth algebraic space. Let $Z_{d} \subset Z$ be an assignment (so-called \emph{smooth part of $Z$}) where $v \in Z_{d}(G)$ if and only if there is $H \in S_{2}(n)$ and $w \in Z(H)$ such that $G \rightsquigarrow H$ and $v \rightsquigarrow w$. We claim that $Z_{d}$ is a smooth extremal assignment. For $G \in S_{2}(n)$, $Z(G) = Z_{d}(G)$. So by Lemma \ref{lem:assignmenttocontractionindicator}, $Z_{d}$ on $S_{2}(n)$ defines a contraction indicator. By Proposition \ref{prop:divassignment}, it is uniquely extended to a smooth extremal assignment $Z_{d}$.

The contraction map $\pi_{Z} : \Mznb \to \Mznb(Z)$ factors through $\Mznb \to \Mznb(Z_{d}) \to \Mznb(Z)$. If $\Mznb(Z) \ne \Mznb(Z_{d})$, then $\Mznb(Z)$ is a small contraction of $\Mznb(Z_{d})$, because it contracts an internal component with at least 4 special points, so the codimension of the exceptional set in $\Mznb(Z_{d})$ is at least 2. Note that $\Mznb(Z_{d})$ is a smooth space by the first part of the proof. A small contraction of a smooth space is singular. Therefore $\Mznb(Z_{d}) = \Mznb(Z)$. Thus $Z_{d}$ is equivalent to $Z$ by Proposition \ref{prop:bijectiveimpliesequivalence}.
\end{proof}

Because of Lemma \ref{lem:divisorialdeterminedbyS2}, a smooth extremal assignment can be described by simpler combinatorial data. There is a correspondence between the set of smooth extremal assignments of order $n$ and a set of well-studied combinatorial objects in hypergraph theory. We refer to \cite{Ber89} for notations on hypergraphs. 

\begin{definition}
A \emph{simple intersecting family} of order $n$, rank $\le n-2$, antirank $\ge 2$ is a hypergraph $H = (V(H), E(H))$ whose vertex set is $V(H) = [n]$ and whose edge set is $E(H) = \{A_{i}\}$, and that satisfies:
\begin{enumerate}
\item For any $A_{i} \in E(H)$, $2 \le |A_{i}| \le n-2$;
\item $H$ is simple (i.e., $A_{i} \subset A_{j} \Rightarrow i = j$);
\item $A_{i} \cap A_{j} \ne \emptyset$ for any $A_{i}, A_{j} \in E(H)$.
\end{enumerate}
\end{definition}

\begin{theorem}\label{thm:divisorialsimpleintersecting}
There is a bijection between the set of smooth extremal assignments of order $n$ and the set of simple intersecting families of order $n$, rank $\le n-2$., antirank $\ge 2$. 
\end{theorem}

\begin{proof}
For a smooth extremal assignment $Z$, we can construct a contraction indicator $\cC_{Z}$. Let $\cC_{Z}^{M} = \{B\in \cC_{Z}\;|\; B \subset B' \Rightarrow B = B'\}$, that is, the set of maximal elements. Then let $\cE_{Z} = \{B^{c}\;|\; B \in \cC_{Z}^{M}\}$. Then it is straightforward to check that $\cE_{Z}$ is a simple intersecting family. Obviously this construction is reversible. 
\end{proof}

\subsection{Projectivity of contractions}

Although $\Mznb(Z)$ is a smooth divisorial contraction of a smooth projective variety $\Mznb$ when $Z$ is smooth, in general it is not a projective variety. In particular, it does not appear in Mori's program (or the log minimal model program) for $\Mznb$. In this section, we show the projectivity of some $\Mznb(Z)$. 

\begin{lemma}\label{lem:blowup}
Let $Z_{1}, Z_{2}$ be two smooth extremal assignments such that $Z_{1} \subset Z_{2}$. Let $\cC_{i}$ be the corresponding contraction indicator for $Z_{i}$. Suppose that either one of the following two conditions holds.
\begin{enumerate}
\item $\cC_{2} = \cC_{1} \cup \{B' \;|\; B' \subset B, |B'| \ge 2\}$ where $|B| = 3$;
\item $\cC_{2}\setminus \cC_{1} = \{B\}$ with $|B| \ge 4$.
\end{enumerate}
Then the reduction map $\pi_{Z_{1}, Z_{2}} : \Mznb(Z_{1}) \to \Mznb(Z_{2})$ is a smooth blow-down. 
\end{lemma}

\begin{proof}
We denote the image $\pi_{Z_{i}}(D_{B})$ by $D_{B}$ if the image is a divisor. For the first case, consider $D_{B} \subset \Mznb(Z_{1})$. Then $D_{B} \cong \overline{\mathrm{M}}_{0, 4}(Z_{1}^{B}) \times \overline{\mathrm{M}}_{0, n-2}(Z_{1}^{B^{c}}) \cong \PP^{1} \times \overline{\mathrm{M}}_{0,n-2}(Z_{1}^{B^{c}})$, and $\pi_{Z_{1}, Z_{2}}(D_{B}) \cong \overline{\mathrm{M}}_{0, n-2}(Z_{1}^{B^{c}})$ (Remark \ref{rem:contractedlocus}). We compute the restriction of the normal bundle along the fiber $\PP^{1}$ containing $x \in D_{B}$. To calculate it, it suffices to pick a general point $X = (C = C_{1} \cup C_{2}, x_{1}, x_{2}, \cdots, x_{n})$ where $C_{1} \in \overline{\mathrm{M}}_{0, 4}(Z_{1}^{B})$ and $C_{2} \in \mathrm{M}_{0, n-2}$. In this case, the normal bundle to $D_{B}$ is isomorphic to that in $\Mznb$. Thus the restriction is $\cO(-1)$. Therefore, by the Fujiki-Nakano criterion (\cite{FN71}), it is a blow-down with 1-dimensional exceptional fibers, in the category of Moishezon manifolds, or equivalently, that of algebraic spaces.

In the second case, the exceptional set $D_{B} \subset \Mznb(Z_{1})$ is isomorphic to $\overline{\mathrm{M}}_{0, |B|+1}(Z_{1}^{B})\times \overline{\mathrm{M}}_{0, n - |B| + 1}(Z_{1}^{B^{c}})$. The extremal assignment $Z_{1}^{B}$ is on the label set $B \cup \{p\}$. Note that every subset $C \subset B \cup \{p\}$ such that $|C| = |B| - 1$ and $p \notin C$ is assigned. Therefore $Z_{1}^{B} = Z_{A}$ with $A = ((\frac{1}{|B|-1})^{|B|},1)$, where the last weight is for $p$. Then $\overline{\mathrm{M}}_{0, |B|+1}(Z_{1}^{B}) \cong \Mza \cong \PP^{|B|-2}$ (\cite[Section 6.1]{Has03}). If we take a general point $X = (C = C_{1} \cup C_{2}, x_{1}, x_{2}, \cdots, x_{n}) \in D_{B}$ such that $C_{2} \in \mathrm{M}_{0, n - |B| + 1}$, then locally the normal bundle to $D_{B}$ is isomorphic to the normal bundle to $D_{B}$ in $\overline{\mathrm{M}}_{0, A'}$, where $A' = ((\frac{1}{|B|-1})^{|B|}, 1^{n-|B|})$. In this case, the restriction of the normal bundle to a fiber $\PP^{|B|-2}$ is $\cO(-1)$. Therefore this is a smooth blow-down, too.
\end{proof}

\begin{proposition}\label{prop:blowup}
Let $Z, Z'$ be two smooth extremal assignments such that $Z \subsetneq Z'$. Then the reduction map $\pi_{Z, Z'} : \Mznb(Z) \to \Mznb(Z')$ is decomposed into a sequence of blow-downs. 
\end{proposition}

\begin{proof}
We will construct a sequence of contraction indicators $\cC_{0}, \cC_{1}, \cdots, \cC_{m}$ such that 
\begin{enumerate}
\item $\cC_{i} \subset \cC_{i+1}$,
\item $\cC_{0}$ corresponds to $Z$ and $\cC_{m}$ corresponds to $Z'$, 
\item $\cC_{i}$ and $\cC_{i+1}$ satisfy one of two conditions in Lemma \ref{lem:blowup}.
\end{enumerate}

For any $B \in \cC_{0}$, there is $B' \in \cC_{m}$ such that $B \subset B'$. Suppose that $B \subsetneq B'$. Take a minimal $D \subset B'$ such that $D$ is not in $\cC_{0}$ and $|D| \ge 3$. If $|D| = 3$, then set $\cC_{1} := \cC_{0} \cup \{D'\;|\; D' \subset D, |D'| \ge 2\}$. If $|D| \ge 4$, set $\cC_{1} := \cC_{0} \cup \{D\}$. We can repeat this procedure until we reach $\cC_{m}$. 

This implies that there is a sequence of smooth extremal assignments $Z = Z_{0} \subset Z_{1} \subset \cdots \subset Z_{m} = Z'$ such that 
\[
	\Mznb(Z) = \Mznb(Z_{0}) \to \Mznb(Z_{1}) \to \cdots \to 
	\Mznb(Z_{m}) = \Mznb(Z')
\]
is a sequence of smooth blow-downs. 
\end{proof}

\begin{corollary}\label{cor:projective}
Let $Z$ be a smooth extremal assignment and let $A$ be weight data such that $Z \subset Z_{A}$. Then $\Mznb(Z)$ is a smooth projective variety.
\end{corollary}

\begin{proof}
By Proposition \ref{prop:blowup}, we can find a sequence of smooth extremal assignments $Z = Z_{0} \subset Z_{1} \subset \cdots \subset Z_{m} = Z_{A}$ such that 
\[
	\pi : \Mznb(Z) = \Mznb(Z_{0}) \to \Mznb(Z_{1}) \to \cdots 
	\to \Mznb(Z_{m}) = \Mznb(Z_{A})
\]
is a sequence of smooth blow-ups. Since $\Mznb(Z_{A}) = \Mza$ is a projective variety and a blow-up of a projective variety is also projective, we obtain the result.
\end{proof}

\begin{example}\label{ex:noweightassignment}
There is a smooth extremal assignment $Z$ without any larger weight assignment. Let $\{\{1,3,4,5\},\{2,4,5,6\},\{1,5,6,7\},\{2,3,5,7\}\}$ be the set of maximal elements in the contraction indicator for $Z$. If there is a weight $A = (a_{1}, a_{2}, \cdots, a_{7})$ such that $Z \subset Z_{A}$, we have
\[
	a_{1}+a_{3}+a_{4}+a_{5} \le 1, a_{2} + a_{4} + a_{5} + a_{6}\le 1, 
	a_{1}+a_{5}+a_{6}+a_{7} \le 1, a_{2}+a_{3}+a_{5}+a_{7}\le 1.
\]
By adding all of them, we have
\[
	2(\sum_{i=1}^{7}a_{i}) + 2a_{5} \le 4.
\]
But since $\sum_{i=1}^{7}a_{i} > 2$, we have $a_{5} < 0$, which is impossible.
\end{example}

In Section \ref{sec:nonprojectiveexamples}, we present some non-projective examples. 

We leave three sufficient conditions for the existence of a larger weight assignment, which guarantees the projectivity of the birational model.

\begin{proposition}\label{prop:existenceweightassignment}
Let $Z$ be a smooth extremal assignment and let $\cC = \{B_{i}\}_{1 \le i \le k}$ be the corresponding contraction indicator. $Z$ has a larger weight assignment $Z_{A}$ under any of following conditions:
\begin{enumerate}
\item There is a label $j \in [n]$ such that $j \notin \bigcup_{i=1}^{k}B_{i}$.
\item There is only one $B_{i}$ containing $j$. 
\item $|B_{i}| < n/2$ for all $i$.
\end{enumerate}
\end{proposition}

\begin{proof}
For each case, we construct weight data $A = (a_{1}, a_{2}, \cdots, a_{n})$ with the property $\sum_{k \in B_{i}}a_{k} \le 1$. For case (1), assume that $j = 1$. Then $A = (1, \left(\frac{1}{n-2}\right)^{n-2})$ works. For (2), assume that $j = 1$ and $1 \in B_{1}$. Let $\ell := |B_{1}|$. Set 
\[
	a_{i} = \begin{cases}1-(\ell-1)\epsilon, & i = 1\\
	\epsilon, & i \in B_{1} - \{1\}\\
	\frac{1}{n-\ell}+\epsilon, & i \notin B_{1}
	\end{cases}
\]
for $0 < \epsilon \ll 1$. Finally, for (3), $A = (\left(\frac{2}{n}+\epsilon\right)^{n})$ works. 
\end{proof}


\section{Toric birational models and graph associahedra}\label{sec:toric}

In \cite{LM00}, Losev and Manin introduced a modular birational model $\Mznb^{LM}$ of $\Mznb$, which is a smooth projective \emph{toric} variety whose corresponding polytope is the $(n-3)$-dimensional permutohedron. This space is called the \emph{Losev-Manin space} and it parametrizes $n$-pointed chains of rational curves. Later Hassett described it as $\Mza$ with $A = (\left(\frac{1}{n-2}\right)^{n-2},1,1)$ (\cite{Has03}). In particular, there is a reduction map $\Mznb^{LM} \to \PP^{n-3}$ if we interpret $\PP^{n-3}$ as $\overline{\mathrm{M}}_{0, (\left(\frac{1}{n-2}\right)^{n-1},1)}$ (\cite[Section 6.1, 6.4]{Has03}). 

The Losev-Manin space was generalized as toric varieties associated to graph assciahedra. Recall that for any rational $n$-dimensional polytope $P$, one can associate a complete toric variety $X(P)$ (\cite{Ful93}). In particular, for any connected graph $G$ with $n-2$ vertices, one can associate a toric variety $X(PG)$ as the following. 

\begin{definition}
Let $G$ be a connected graph. A \emph{tube} of $G$ is a subset $B \subset V(G)$ which generates a connected subgraph of $G$. A tube $B$ is trivial if $|B| = 1$. 
\end{definition}

Note that an edge of $G$ is a nontrivial tube. 

Let $G$ be a connected graph such that $V(G) = [n]$. Then we can associate an $(n-1)$-dimensional polytope $PG$, a so-called \emph{graph associahedron}. We start with an $(n-1)$-dimensional simplex $PG^{(0)}$ with $n$ facets indexed by $[n]$. We can identify higher codimensional faces with subsets of $[n]$. Let $T_{i}$ be the set of tubes in $G$ of size $n-i$. Let $PG^{(1)}$ be the polytope obtained by truncating faces in $T_{n-1}$ (vertices). Then $PG^{(i)}$ is obtained by truncating faces of $PG^{(i-1)}$ in $T_{n-i}$. Let $PG = PG^{(n-2)}$ (See \cite[Construction 1]{dRJR14}). 

In particular, if $G = K_{n-2}$, $X(PG) = \Mznb^{LM}$ (\cite[Section 6.4]{Has03}). 

In \cite[Theorem 1]{dRJR14}, it was shown that $X(PG) \cong \Mza$ if and only if $G$ is an iterated cone over a discrete set. In this section, we show the following quick application of the theory of smooth extremal assignments. 

\begin{definition}
A connected graph $G$ is \emph{co-transitive} if for any edge $e$ and a vertex $v$, there is an edge $f$ that connects $v$ and one end point of $e$. 
\end{definition}

Lemma \ref{lem:cotransitivemultipartite} is due to Andreas Blass and Daniel Solt\'esz.

\begin{lemma}\label{lem:cotransitivemultipartite}
A connected graph $G$ is co-transitive if and only if it is a complete multipartite graph. 
\end{lemma}

\begin{proof}
If $G$ is co-transitive, then non-adjacency is transitive: If $v$ and $w$ are not adjacent and $w$ and $x$ are not adjacent, then $v$ and $x$ are not adjacent. Thus the complement graph of $G$ is a disjoint union of cliques. Therefore $G$ is complete multipartite. The converse is easy to check.
\end{proof}

\begin{lemma}\label{lem:connectingpropertytube}
Let $G$ be a co-transitive graph. Then for any subset $B'$ of a non-tube, $B$ is also a non-tube.
\end{lemma}

\begin{proof}
If $B'$ is a tube, then by using co-transitivity, we can enlarge $B'$ to a tube $B$. 
\end{proof}

\begin{theorem}
For any connected graph $G$ of $(n-2)$ vertices, $X(PG) \cong \Mznb(Z)$ for some smooth extremal assignment $Z$ if and only if $G$ is a complete multipartite graph. 
\end{theorem}

\begin{proof}
Note that $PG$ is obtained from an $(n-3)$-dimensional simplex by truncating faces in an increasing order of dimension. Dually, $X(PG)$ is obtained by taking successive blow-ups of torus invariant subvarieties of $\PP^{n-3}$ in an increasing order of dimension. 

Let $T$ be the set of tubes in $G$. Let 
\[
	\cC := \{B \cup \{n-1\} \subset [n-1]\;|\; B \notin T, |B| \ge 2\}
	\bigcup \{B' \subset [n-2]\;|\;|B'| \ge 2\}
	\subset [n].
\]
This set $\cC$ indexes the set of contracted divisors. 

We claim that $\cC$ is a contraction indicator (Definition \ref{def:contractionindicator}) if and only if $G$ is co-transitive. If this is true, then $X(PG)$ is a smooth blow-up of $\PP^{n-3} = \overline{\mathrm{M}}_{0, ((1/(n-2))^{n-1}.1)}$ and it can be interpreted as $\Mznb(Z)$, as in Proposition \ref{prop:blowup}.

Items (1) and (3) of Definition \ref{def:contractionindicator} are clear from the definition, since $n \notin B$ for all $B \in \cC$. So it suffices to check item (2).

Suppose that $G$ is co-transitive. Assume that $B \in \cC$ and $B' \subset B$ and $|B'| \ge 2$. If $B' \subset [n-2]$, then by definition, $B' \in \cC$. If $n-1 \in B'$, then $B'' := B' \setminus \{n-1\}$ is a subset of a non-tube $B \setminus \{n-1\}$. Therefore $B''$ is also a non-tube by Lemma \ref{lem:connectingpropertytube} and $B' \in \cC$.

Conversely, suppose that $\cC$ is a contraction indicator. Let $e = \{i, j\}$ be an edge of $G$ and let $k$ be a vertex not on $e$. If there is no edge connecting $e$ and $k$, then $\{i, j, k\}$ is not a tube. Then $\{i, j, k, n-1\} \in \cC$. Since $\cC$ is a contraction indicator, $\{i, j, n-1\} \in \cC$. So $\{i, j\}$ is not a tube, which is a contradiction.Therefore $G$ is co-transitive.
\end{proof}

For any connected graph $G$ of order $n-2$, the corresponding toric variety $X(PG)$ is smooth, and it is a divisorial contraction of $X(PK_{n-2}) = \Mznb^{LM}$. However, only very few of them have a moduli theoretic interpretation. This is further evidence that the current moduli-theoretic description of birational models is insufficient to understand birational geometric behavior of $\Mznb$.


\section{$S_{n}$-invariant extremal assignments}\label{sec:invariant}

There is a natural $S_{n}$-action on $\Mznb$, permuting $n$ labels. Let $\Mznt$ be the quotient space $\Mznb/S_{n}$. This space has a natural moduli-theoretic interpretation as a moduli space of stable rational curves with unordered marked points, and the birational geometric properties are relatively simple compared to that of $\Mznb$. For instance, the Picard number $\rho(\Mznt)$ is $\lfloor n/2\rfloor - 1$, and the effective cone is simplicial. It is expected that $\Mznt$ is a Mori dream space, in contrast to $\Mznb$, which is not a Mori dream space for large $n$ (\cite{CT15}).

The $S_{n}$-action naturally induces $S_{n}$-actions on the set of stable $n$-labeled trees $S(n)$, the set of set partitions $\cP(n)$, and the set of all (extremal) assignments of order $n$. Let $Z$ be an $S_{n}$-invariant extremal assignment. Then the contraction map $\pi_{Z} : \Mznb \to \Mznb(Z)$ is $S_{n}$-equivariant. Therefore we have a quotient map $\bar{\pi}_{Z} : \Mznt  \to \Mznb(Z)/S_{n}$. Thus $S_{n}$-invariant extremal assignments give contractions of $\Mznt$. In this section, we study $S_{n}$-invariant extremal assignments.

In Section \ref{sec:structuretheorem}, we have seen that the structure of an extremal assignment can be described in terms of set partitions or, equivalently, basic pairs. Let $(G, v)$ be a basic pair and suppose that $v \in Z(G)$ for an $S_{n}$-invariant extremal assignment $Z$. Since $Z$ is $S_{n}$-invariant, for any permutation of $n$ labels of $G$, the central vertex must be assigned too. Therefore the fact that a central vertex is assigned or not depends on the topological type of the underlying graph, not a specific labeling. Thus in the case of $S_{n}$-invariant assignments, a natural description is obtained in terms of integer partitions. This is a crucial difference between ordinary extremal assignments and invariant extremal assignments, because there are nontrivial automorphisms for the underlying tree without labeling. 

Let $\cI(n)$ be the set of integer partitions of $n$. For $p, q \in \cI(n)$, we say $p \le q$ if $q$ is a finer partition than $p$. Then $\cI(n)$ is a partially ordered set. The maximum element is the complete partition $\{1, 1, \cdots, 1\}$, and the minimum element is $\{n\}$. We can define refinements, corruptions, and upper bounds of two integer partitions in a similar way. 

Recall that $\cP(n)$ is the set of set partitions of $[n]$. There is a natural forgetful map 
\[
	f : \cP(n) \to \cI(n),
\]
which maps $P = \{B_{1}, B_{2}, \cdots, B_{k}\}$ to the partition $f(P) = \{|B_{1}|, |B_{2}|, \cdots, |B_{k}|\}$. Then $f$ is an $S_{n}$-invariant order-preserving map. 

The following lemma explains relations between set partitions and integer partitions. The proof is straightforward.

\begin{lemma}\label{lem:comparisonofupperbound}
Let $P, Q \in \cP(n)$. 
\begin{enumerate}
\item Suppose that $P \ne Q$ and $p = f(P)$, $q = f(Q)$. Let $p = \{\lambda_{1}, \lambda_{2}, \cdots, \lambda_{s}\}$ and $q = \{\mu_{1}, \mu_{2}, \cdots, \mu_{t}\}$. Then $P$ and $Q$ have a tight upper bound only if there are $i$ and $J$ with $|J| \ge 2$ such that $\lambda_{i} = \sum_{j \in J}\mu_{j}$ or vice versa.
\item If $f(P) \le f(Q)$, then there is $\sigma \in S_{n}$ such that $P \le \sigma \cdot Q$. 
\end{enumerate}
\end{lemma}

As in the case of ordinary extremal assignments, one may try to define `atomic' $S_{n}$-invariant extremal assignments. However, in contrast to the case of ordinary extremal assignments, it is not true that for any basic pair $(G, v)$, there is an $S_{n}$-invariant extremal assignment $Z$ such that $v \in Z(G)$. 

\begin{example}
Let $(G,v)$ a basic pair with corresponding set partition $P = \{\{1,2,3\},\{4\},\{5\},\{6\}\}$, and let $Z$ be an $S_n$-invariant extremal assignment so that $v \in Z(G)$. We can permute the labels in $P$ to obtain another basic pair $(H,w)$ with set partition $P' = \{\{1\},\{2\},\{3\},\{4,5,6\}\}$. If $v\in Z(G)$, then $w\in Z(H)$, by definition of an $S_n$-invariant extremal assignment. We can see that indeed $H = G$ and $v$ and $w$ are two vertices of $G$, so $Z(G) = V(G)$. Thus this basic pair has no $S_n$-invariant extremal assignment.
\end{example}

Let $Z_{G}$ be the atomic extremal assignment associated to $(G, v)$. Then by Proposition \ref{prop:atomic}, $Z_{G} = \bigcup_{Q \preceq P, |Q| \ge 3}Z_{Q}$, where $P$ is the set partition corresponding to $(G, v)$. It is clear that any $S_{n}$-invariant extremal assignment contains $\bigcup_{\sigma \in S_{n}}\sigma \cdot Z_{G} = \bigcup_{\sigma \in S_{n}}Z_{\sigma \cdot G}$. 

\begin{definition}\label{def:specialpartition}
Let $p = \{\lambda_{1}, \lambda_{2}, \cdots, \lambda_{k}\}$ be an integer partition of $n$. 
\begin{enumerate}
\item A \emph{replacement} of $p$ is an integer partition of $[n]$ which is obtained by performing the following operation several times: Remove a subpartition $\{\lambda_{i}\}_{i \in I} \subset p$ and add another subpartition $\{\lambda_{j}\}_{j \in J} \subset p$ for two index sets $I, J$ such that $\sum_{i \in I}\lambda_{i} = \sum_{j \in J}\lambda_{j}$. 
\item An integer partition $\{\lambda_{1}, \lambda_{2},\cdots, \lambda_{k}\}$ of $n$ is \emph{special} if for any replacement $r$ of $p$, $r \le p$. 
\item For a family $\cF := \{p_{1}, p_{2}, \cdots, p_{m}\}$ of integer partitions of $[n]$, a \emph{replacement in $\cF$} is an integer partition of $[n]$ which is obtained by performing the following operation several times: Start from $q \le p_{s}$. Remove a subpartition $\{\lambda_{i}\}_{i \in I} \subset p_{s}$ and add another subpartition $\{\mu_{j}\}_{j \in J} \subset p_{t}$ for two index sets $I, J$ such that $\sum_{i \in I}\lambda_{i} = \sum_{j \in J}\mu_{j}$. 
\item A family $\cF$ is called a \emph{special family} if for any replacement $r$ in $F$, $r \le p_{j}$ for some $p_{j} \in \cF$.
\end{enumerate}
\end{definition}

Thus a singleton famliy $\cF = \{p\}$ is special if and only if $p$ is special. 

\begin{example}
\begin{enumerate}
\item Let $p_1 = \{4,3,2,1\}$. Since $4+1 = 3+2$, $r_1 = \{4,4,1,1\}$ and $r_2 = \{3,3,2,2\}$ are replacements of $p_1$. Since $r_2 \not\leq p_1$,  $p_1$ is not special.
\item Let $p_2 = \{6,6,1,1\}$. Clearly the only replacement of $p_2$ is itself, so $p_2$ is special.
\item Let $\cF = \{p_1,p_2,p_3\}$ where $p_1 = \{3,3,3,3,2\}$, $p_2 = \{3,3,2,2,2,2\}$, and $p_3 = \{2,2,2,2,2,2,2\}$. It is easy to check that for any replacement $r$ in $\cF$, $r \le p_1$, $p_2$, or $p_3$, so $\cF$ is a special family.
\item Let $\cF = \{p_1,p_2\}$ where $p_1 = \{5,5,2,2\}$ and $p_2 = \{4,4,3,3\}$. There is a replacement in $\cF$, $r = \{5,4,3,2\}$, such that $r \not\leq p_1$, $r \not\leq p_2$. Therefore $\cF$ is not a special family.
\end{enumerate}
\end{example}

\begin{theorem}\label{thm:invariantstructurethm}
Let $(G_{1}, v_{1}), (G_{2}, v_{2}), \cdots, (G_{k}, v_{k})$ be basic pairs and let $p_{1}, p_{2}, \cdots, p_{k}$ be the corresponding integer partitions. Let $Z_{p_{i}} := \bigcup_{\sigma \in S_{n}}\sigma \cdot Z_{G_{i}}$ where $Z_{G_{i}}$ is the atomic extremal assignment associated to $(G, v)$. Then $Z := \bigcup_{i=1}^{k}Z_{p_{i}}$ is an $S_{n}$-invariant extremal assignment if and only if $\cF := \{p_{1}, p_{2}, \cdots, p_{k}\}$ is a special family.
\end{theorem}

\begin{proof}
Obviously $Z$ is $S_{n}$-invariant. Thus it suffices to check the extremality. Let $P_{i}$ be the set partition corresponding to $(G_{i}, v_{i})$. 

Suppose that $\cF$ is a special family. Let $Q_{1} \preceq \sigma_{i}\cdot P_{i}$ and $Q_{2} \preceq \sigma_{j}\cdot P_{j}$. And let $p_{i} = f(P_{i}) = \{\lambda_{1}, \lambda_{2}, \cdots, \lambda_{s}\}$ and $p_{j} = f(P_{j}) = \{\mu_{1}, \mu_{2}, \cdots, \mu_{t}\}$. Suppose that $Q_{1}, Q_{2}$ have a tight upper bound $R$ and $r := f(R)$. Let $q_{1} := f(Q_{1}) = \{\nu_{1}, \nu_{2}, \cdots, \nu_{u}\}$ and $q_{2} := \{\xi_{1}, \xi_{2}, \cdots, \xi_{v}\}$. By (1) of Lemma \ref{lem:comparisonofupperbound}, there are several pairs $(i, J)$ with $|J| \ge 2$ such that $\nu_{i} = \sum_{j \in J}\xi_{j}$ (or vice versa). Since $q_{2}\le p_{j}$, $\nu_{i} = \sum_{j \in J}\xi_{j} = \sum_{k \in K}\mu_{k}$ for some $K$. Thus $r$ is a replacement of $q_{1}$ in $\cF$. Since $\cF$ is a special family, $r \le p_{k}$ for some $k$. By (2) of Lemma \ref{lem:comparisonofupperbound}, there is $\tau \in S_{n}$ such that $R \le \tau \cdot P_{k}$. By Theorem \ref{thm:structurethm}, $Z$ is extremal. 

Conversely, assume that $Z$ is extremal. We claim that for any replacement $r$ in $\cF$ and a corresponding basic pair $(H, w)$, $w \in Z(H)$. By induction, it is sufficient to show for a replacement $r$ which is obtained by one operation. Let $q \le p_{i} = \{\lambda_{1}, \lambda_{2}, \cdots, \lambda_{s}\}$ and $r$ be obtained by replacing $\{\lambda_{i}\}_{i \in I} \subset p_{i}$ and adding $\{\mu_{j}\}_{j \in J} \subset p_{j}$. Consider the following list of degenerations/contractions (Figure \ref{fig:degcont}). 
\begin{enumerate}
\item Start with $(G_{0}, v_{0})$, which corresponds to $q$. $v_{0} \in Z(G_{0})$. Degenerate $v_{0}$ to two vertices $v_{1}, v_{1}'$ and obtain $G_{1}$. $v_{1}, v_{1}' \in Z(G_{1})$. Here $v_{1}$ has tails indexed by $I$.
\item Contract $v_{1}$ and the tails connected to $v_{1}$ to a vertex $v_{2}$ and obtain $(G_{2}, v_{2}')$. $v_{2}' \in Z(G_{2})$. 
\item Degenerate $v_{2}$ and get $G_{3}$ where $v_{3}$ has tails indexed by $J$. $v_{3}' \in Z(G_{3})$.
\item Contract $v_{3}'$ and its tails to a vertex and obtain $(G_{4}, v_{4})$. Since it corresponds to a partition $s \le p_{j}$, $v_{4} \in Z(G_{4})$.
\item On $G_{3}$, $v_{3} \in Z(G_{3})$ from (4). 
\item By contracting $v_{3}, v_{3}'$ to $v_{5}$, we obtain $G_{5}$. $v_{5} \in Z(G_{5})$. $G_{5}$ corresponds to the replacement $r$. 
\end{enumerate}

\begin{figure}[!ht]
\begin{tikzpicture}[scale=0.35]
\draw[line width = 1 pt] (5, 5) -- (3, 3.2);
\draw[line width = 1 pt] (5, 5) -- (2.6, 5);
\draw[line width = 1 pt] (5, 5) -- (3, 6.8);
\draw[line width = 1 pt] (5, 5) -- (7, 7.2);
\draw[line width = 1 pt] (5, 5) -- (7.3, 5);
\draw[line width = 1 pt] (5, 5) -- (7, 2.9);
\draw[line width = 1 pt] (3, 3.2) -- (2.1, 3.2);
\draw[line width = 1 pt] (3, 3.2) -- (2.2, 2.8);
\draw[line width = 1 pt] (3, 3.2) -- (2.4, 2.5);
\draw[line width = 1 pt] (2.6, 5) -- (1.9, 5.6);
\draw[line width = 1 pt] (2.6, 5) -- (1.7, 5.2);
\draw[line width = 1 pt] (2.6, 5) -- (1.7, 4.8);
\draw[line width = 1 pt] (2.6, 5) -- (1.9, 4.5);
\draw[line width = 1 pt] (3, 6.8) -- (2.1, 7);
\draw[line width = 1 pt] (3, 6.8) -- (2.5, 7.7);
\draw[line width = 1 pt] (7.3, 5) -- (8.2, 5.2);
\draw[line width = 1 pt] (7.3, 5) -- (8, 5.6);
\draw[line width = 1 pt] (7.3, 5) -- (8.2, 4.8);
\draw[line width = 1 pt] (7.3, 5) -- (8, 4.4);
\draw[line width = 1 pt] (7, 7.2) -- (7.4, 8);
\draw[line width = 1 pt] (7, 7.2) -- (7.8, 7.6);
\draw[line width = 1 pt] (7, 2.9) -- (7.7, 2.5);
\draw[line width = 1 pt] (7, 2.9) -- (7.4, 2.2);
\fill (5, 5) circle (8pt);
\node (v) at (5, 4.2) {$v_{0}$};
\node (G) at (5, 2.5) {$G_{0}$};
\node (a) at (9.3, 5) {$\leftsquigarrow$};
\end{tikzpicture}
\begin{tikzpicture}[scale=0.35]
\draw[line width = 1 pt] (5, 5) -- (3, 3.2);
\draw[line width = 1 pt] (5, 5) -- (2.6, 5);
\draw[line width = 1 pt] (5, 5) -- (3, 6.8);
\draw[line width = 1 pt] (5, 5) -- (7, 5);
\draw[line width = 1 pt] (7, 5) -- (9, 7.2);
\draw[line width = 1 pt] (7, 5) -- (9.3, 5);
\draw[line width = 1 pt] (7, 5) -- (9, 2.9);
\draw[line width = 1 pt] (3, 3.2) -- (2.1, 3.2);
\draw[line width = 1 pt] (3, 3.2) -- (2.2, 2.8);
\draw[line width = 1 pt] (3, 3.2) -- (2.4, 2.5);
\draw[line width = 1 pt] (2.6, 5) -- (1.9, 5.6);
\draw[line width = 1 pt] (2.6, 5) -- (1.7, 5.2);
\draw[line width = 1 pt] (2.6, 5) -- (1.7, 4.8);
\draw[line width = 1 pt] (2.6, 5) -- (1.9, 4.5);
\draw[line width = 1 pt] (3, 6.8) -- (2.1, 7);
\draw[line width = 1 pt] (3, 6.8) -- (2.5, 7.7);
\draw[line width = 1 pt] (9.3, 5) -- (10.2, 5.2);
\draw[line width = 1 pt] (9.3, 5) -- (10, 5.6);
\draw[line width = 1 pt] (9.3, 5) -- (10.2, 4.8);
\draw[line width = 1 pt] (9.3, 5) -- (10, 4.4);
\draw[line width = 1 pt] (9, 7.2) -- (9.4, 8);
\draw[line width = 1 pt] (9, 7.2) -- (9.8, 7.6);
\draw[line width = 1 pt] (9, 2.9) -- (9.7, 2.5);
\draw[line width = 1 pt] (9, 2.9) -- (9.4, 2.2);
\fill (5, 5) circle (8pt);
\fill (7, 5) circle (8pt);
\node (v) at (5, 4.2) {$v_{1}$};
\node (v) at (7, 4.2) {$v_{1}'$};
\node (G) at (6, 2.5) {$G_{1}$};
\node (a) at (11.3, 5) {$\rightsquigarrow$};
\end{tikzpicture}
\begin{tikzpicture}[scale=0.35]
\draw[line width = 1 pt] (5, 5) -- (7, 5);
\draw[line width = 1 pt] (7, 5) -- (9, 7.2);
\draw[line width = 1 pt] (7, 5) -- (9.3, 5);
\draw[line width = 1 pt] (7, 5) -- (9, 2.9);
\draw[line width = 1 pt] (5, 5) -- (4, 3.2);
\draw[line width = 1 pt] (5, 5) -- (3.9, 6);
\draw[line width = 1 pt] (5, 5) -- (4, 6.4);
\draw[line width = 1 pt] (5, 5) -- (3.9, 3.6);
\draw[line width = 1 pt] (5, 5) -- (3.9, 3.9);
\draw[line width = 1 pt] (5, 5) -- (3.7, 4.8);
\draw[line width = 1 pt] (5, 5) -- (3.8, 4.5);
\draw[line width = 1 pt] (5, 5) -- (3.7, 5.4);
\draw[line width = 1 pt] (5, 5) -- (3.8, 5.7);
\draw[line width = 1 pt] (9.3, 5) -- (10.2, 5.2);
\draw[line width = 1 pt] (9.3, 5) -- (10, 5.6);
\draw[line width = 1 pt] (9.3, 5) -- (10.2, 4.8);
\draw[line width = 1 pt] (9.3, 5) -- (10, 4.4);
\draw[line width = 1 pt] (9, 7.2) -- (9.4, 8);
\draw[line width = 1 pt] (9, 7.2) -- (9.8, 7.6);
\draw[line width = 1 pt] (9, 2.9) -- (9.7, 2.5);
\draw[line width = 1 pt] (9, 2.9) -- (9.4, 2.2);
\fill (7, 5) circle (8pt);
\node (v) at (7, 4.2) {$v_{2}'$};
\node (vp) at (5.4, 4.2) {$v_{2}$};
\node (G) at (6, 2.5) {$G_{2}$};
\node (a) at (11.3, 5) {$\leftsquigarrow$};
\end{tikzpicture}
\begin{tikzpicture}[scale=0.35]
\draw[line width = 1 pt] (5, 5) -- (3, 3.8);
\draw[line width = 1 pt] (5, 5) -- (3, 6.2);
\draw[line width = 1 pt] (5, 5) -- (7, 5);
\draw[line width = 1 pt] (7, 5) -- (9, 7.2);
\draw[line width = 1 pt] (7, 5) -- (9.3, 5);
\draw[line width = 1 pt] (7, 5) -- (9, 2.9);
\draw[line width = 1 pt] (3, 3.8) -- (2.1, 3.2);
\draw[line width = 1 pt] (3, 3.8) -- (2.2, 2.8);
\draw[line width = 1 pt] (3, 3.8) -- (2.4, 2.5);
\draw[line width = 1 pt] (3, 3.8) -- (2, 3.6);
\draw[line width = 1 pt] (3, 3.8) -- (2, 4);
\draw[line width = 1 pt] (3, 6.2) -- (1.9, 5.6);
\draw[line width = 1 pt] (3, 6.2) -- (1.7, 6);
\draw[line width = 1 pt] (3, 6.2) -- (1.7, 6.3);
\draw[line width = 1 pt] (3, 6.2) -- (1.9, 6.7);
\draw[line width = 1 pt] (9.3, 5) -- (10.2, 5.2);
\draw[line width = 1 pt] (9.3, 5) -- (10, 5.6);
\draw[line width = 1 pt] (9.3, 5) -- (10.2, 4.8);
\draw[line width = 1 pt] (9.3, 5) -- (10, 4.4);
\draw[line width = 1 pt] (9, 7.2) -- (9.4, 8);
\draw[line width = 1 pt] (9, 7.2) -- (9.8, 7.6);
\draw[line width = 1 pt] (9, 2.9) -- (9.7, 2.5);
\draw[line width = 1 pt] (9, 2.9) -- (9.4, 2.2);
\fill (7, 5) circle (8pt);
\node (v) at (5, 4.2) {$v_{3}$};
\node (v) at (7, 4.2) {$v_{3}'$};
\node (G) at (6, 2.5) {$G_{3}$};
\end{tikzpicture}\\
\begin{tikzpicture}[scale=0.35]
\draw[line width = 1 pt] (5, 5) -- (3, 3.8);
\draw[line width = 1 pt] (5, 5) -- (3, 6.2);
\draw[line width = 1 pt] (5, 5) -- (7, 5);
\draw[line width = 1 pt] (3, 3.8) -- (2.1, 3.2);
\draw[line width = 1 pt] (3, 3.8) -- (2.2, 2.8);
\draw[line width = 1 pt] (3, 3.8) -- (2.4, 2.5);
\draw[line width = 1 pt] (3, 3.8) -- (2, 3.6);
\draw[line width = 1 pt] (3, 3.8) -- (2, 4);
\draw[line width = 1 pt] (3, 6.2) -- (1.9, 5.6);
\draw[line width = 1 pt] (3, 6.2) -- (1.7, 6);
\draw[line width = 1 pt] (3, 6.2) -- (1.7, 6.3);
\draw[line width = 1 pt] (3, 6.2) -- (1.9, 6.7);
\draw[line width = 1 pt] (7, 5) -- (8.2, 5.2);
\draw[line width = 1 pt] (7, 5) -- (8, 5.6);
\draw[line width = 1 pt] (7, 5) -- (8.2, 4.8);
\draw[line width = 1 pt] (7, 5) -- (8, 4.4);
\draw[line width = 1 pt] (7, 5) -- (7.4, 6.4);
\draw[line width = 1 pt] (7, 5) -- (7.8, 6);
\draw[line width = 1 pt] (7, 5) -- (7.7, 4);
\draw[line width = 1 pt] (7, 5) -- (7.4, 3.6);
\fill (5, 5) circle (8pt);
\node (v) at (5, 4.2) {$v_{4}$};
\node (G) at (6, 2.5) {$G_{4}$};
\end{tikzpicture}
\quad
\begin{tikzpicture}[scale=0.35]
\draw[line width = 1 pt] (5, 5) -- (3, 3.8);
\draw[line width = 1 pt] (5, 5) -- (3, 6.2);
\draw[line width = 1 pt] (5, 5) -- (7, 7.2);
\draw[line width = 1 pt] (5, 5) -- (7.3, 5);
\draw[line width = 1 pt] (5, 5) -- (7, 2.9);
\draw[line width = 1 pt] (3, 3.8) -- (2.1, 3.2);
\draw[line width = 1 pt] (3, 3.8) -- (2.2, 2.8);
\draw[line width = 1 pt] (3, 3.8) -- (2.4, 2.5);
\draw[line width = 1 pt] (3, 3.8) -- (2, 3.6);
\draw[line width = 1 pt] (3, 3.8) -- (2, 4);
\draw[line width = 1 pt] (3, 6.2) -- (1.9, 5.6);
\draw[line width = 1 pt] (3, 6.2) -- (1.7, 6);
\draw[line width = 1 pt] (3, 6.2) -- (1.7, 6.3);
\draw[line width = 1 pt] (3, 6.2) -- (1.9, 6.7);
\draw[line width = 1 pt] (7.3, 5) -- (8.2, 5.2);
\draw[line width = 1 pt] (7.3, 5) -- (8, 5.6);
\draw[line width = 1 pt] (7.3, 5) -- (8.2, 4.8);
\draw[line width = 1 pt] (7.3, 5) -- (8, 4.4);
\draw[line width = 1 pt] (7, 7.2) -- (7.4, 8);
\draw[line width = 1 pt] (7, 7.2) -- (7.8, 7.6);
\draw[line width = 1 pt] (7, 2.9) -- (7.7, 2.5);
\draw[line width = 1 pt] (7, 2.9) -- (7.4, 2.2);
\fill (5, 5) circle (8pt);
\node (v) at (5, 4.2) {$v_{5}$};
\node (G) at (5, 2.5) {$G_{5}$};
\end{tikzpicture}
\caption{A sequence of contractions and degenerations}
\label{fig:degcont}
\end{figure}
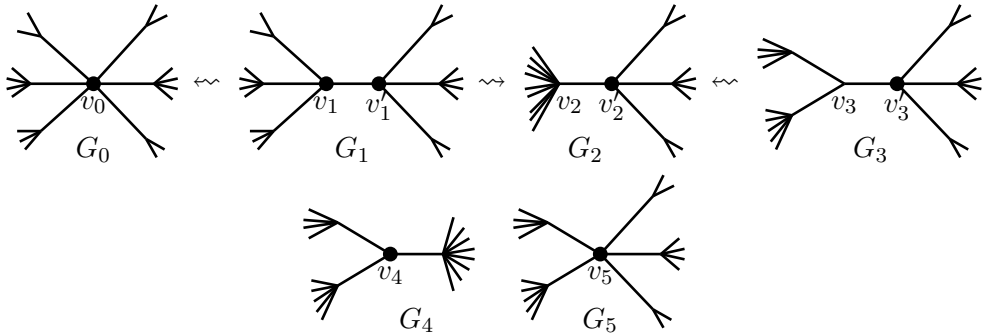

Therefore if $r$ is a replacement, then the corresponding central vertex is assigned. By Proposition \ref{prop:extremaltightbound}, for $R$ with $f(R) = r$, there must be $Z_{P} \subset Z$ so that $R \le P$. Then $f(P) \le p_{k}$ for some $k$ and $r = f(R) \le f(P) \le p_{k}$. Therefore $\cF$ is special.
\end{proof}

The following simple special case of Theorem \ref{thm:invariantstructurethm} is already useful. 

\begin{corollary}
Let $(G, v)$ be a basic pair. Let $Z := \bigcup_{\sigma \in S_{n}}\sigma \cdot Z_{G}$ as above and let $p$ be the corresponding integer partition for $(G, v)$. Then $Z$ is an extremal assignment if and only if $p$ is special.
\end{corollary}

\begin{example}
Let $(G, v)$ be a basic pair with corresponding integer partition $p = \{4,4,3,3\}$, and let $Z := \bigcup_{\sigma \in S_{n}}\sigma \cdot Z_{G}$. Clearly $p$ is special, so $Z$ is an extremal assignment.
\end{example}

Due to Theorem \ref{thm:invariantstructurethm}, we can translate Question \ref{que:fundamentalquestion} for $S_{n}$-invariant extremal assignments in terms of integer partitions. 

\begin{algorithm}[Existence/construction of the smallest extremal assignment]
Let $G_{1}, G_{2}, \cdots, G_{k} \in S(n)$ and $v_{1} \in V(G_{1}), v_{2} \in V(G_{2}), \cdots, v_{k} \in V(G_{k})$. We want to find the smallest $S_{n}$-invariant extremal assignment $Z$ such that $v_{i} \in Z(G_{i})$, if there is one.
\begin{enumerate}
\item Contract each tail of $G_{i}$ adjacent to $v_{i}$ to a single vertex and make a basic pair $(\overline{G}_{i}, v_{i})$. Let $p_{i}$ be the corresponding integer partition.
\item For the family $\cF := \{p_{1}, p_{2}, \cdots, p_{k}\}$, add all replacements of $p_{i}$ in $\cF$ and enlarge $\cF$. 
\item If $\cF$ contains the complete partition $\{1, 1, \cdots, 1\}$, then such $Z$ does not exist.
\item If $\cF$ does not contain the complete partition, then eliminate non-maximal elements in $\cF$. 
\item $Z = \bigcup_{p \in \cF}Z_{p}$ is the smallest $S_{n}$-invariant extremal assignment. 
\end{enumerate}
\end{algorithm}

The algorithm is implemented as a program in Sage. It can be found on the website of the first author:
\begin{center}
	\url{http://www.hanbommoon.net/publications/extremal}
\end{center}

Although there are many new types of $S_{n}$-invariant extremal assignments, all $S_{n}$-invariant smooth assignments are classically known. 

\begin{proposition}\label{prop:invariantdivisorial}
Every $S_{n}$-invariant smooth extremal assignment is a weight assignment $Z_{A}$ for some $S_{n}$-invariant weight $A$.
\end{proposition}

\begin{proof}
Let  $Z$ be an $S_{n}$-invariant smooth extremal assignment, and let $m$ be the largest number of labels on any assigned vertex. For any graph in $S_2(n)$, only the vertex with fewer labels can be assigned, otherwise we could permute the labels to show that the other vertex is assigned as well. So we have $n>2m$. We can simply assign $A=((\frac{1}{m})^n)$. We have $\sum a_i = \frac{n}{m} > \frac{2m}{m} = 2$. So $A$ gives a valid weight assignment.
\end{proof}


\section{Geography of contractions of $\Mznt$}\label{sec:geography}

Currently the two most generalized constructions of modular compactifications of $\Mzn$ are given by 1) extremal assignments and 2) Veronese quotients. They provide two large families of birational contractions of $\Mznb$, large portions of which overlap. Since these two constructions are very general, one may wonder if these two families are sufficient to describe all birational contractions of $\Mznb$. More precisely, we may ask the following question:

\begin{question}\label{que:modularity}
Let $X$ be a projective birational contraction of $\Mznb$. Are there finitely many contractions $\Mznb \to Y_{i}$ such that 
\begin{enumerate}
\item Each $Y_{i}$ is either (the normalization of) $\Mznb(Z)$ or $V_{\gamma, \vec{c}}^{d}$;
\item $X$ is (the normalization of) the image of the product map 
\[
	\Mznb \to \prod_{i=1}^{k}Y_{i}?
\]
\end{enumerate}
\end{question}

\begin{remark}\label{rem:normality}
In projective birational geometry in the sense of Mori's program, if $X$ is a normal projective variety and $Y$ is a projective birational contraction of $X$, then $Y$ is also a normal variety. Indeed, $Y = \proj \bigoplus_{m \ge 0}\rH^{0}(X, mD)$ for some big semi-ample divisor $D$, and the section ring is integrally closed. We do not know if $\Mznb(Z)$ and $V_{\gamma, \vec{c}}^{d}$ are normal or not. Thus in Question \ref{que:modularity}, we must use normalizations of modular contractions. On the other hand, in both cases, the contraction map from $\Mznb$ has a connected fiber, so the normalization $\Mznb(Z)^{\nu}$ (resp. $(V_{\gamma, \vec{c}}^{d})^{\nu}$) is homeomorphic to $\Mznb(Z)$ (resp. $V_{\gamma, \vec{c}}^{d})$. 

From this point, for notational simplicity, every modular compactification will be replaced by its normalization. 
\end{remark}

\begin{remark}
As we have seen in Section \ref{sec:toric}, extremal assignments are not sufficient to construct all projective birational contractions of $\Mznb$. 
\end{remark}

We will see that the answer of Question \ref{que:modularity} is negative even we restrict ourselves to $S_{n}$-invariant contractions, or equivalently, contractions of $\Mznt$. In this section, we compute all faces of the nef cone of $\Mznt$ for small $n$ and try to find the corresponding contractions. 

For $n = 4, 5$, the Picard number $\rho(\Mznt)$ is one. Thus there is no nontrivial contraction. 

\begin{example}[$n = 6$]
In this case, the Picard number $\rho(\widetilde{\mathrm{M}}_{0,6}) = \dim \mathrm{N}^{1}(\widetilde{\mathrm{M}}_{0, 6}) = 2$. $F_{3,1,1,1}$ and $F_{2,2,1,1}$ generate two extremal rays of $\mathrm{NE}_{1}(\widetilde{\mathrm{M}}_{0,6})$. Dually, $\mathrm{Nef}(\widetilde{\mathrm{M}}_{0,6})$ is generated by $D_{2}+3D_{3}$ and $2D_{2}+D_{3}$. The birational model corresponding to $D_{2}+3D_{3}$ is $(\PP^{1})^{6}\git \SL_{2}$, which is classically known as Segre cubic. The model corresponding to $2D_{2}+D_{3}$ is $V_{0, (3/7)^{7}}^{2}$, which is Igusa quartic (\cite{Moo15b}). Both birational models do not come from extremal assignments. Indeed, any $S_{6}$-invariant extremal assignment is equivalent to the empty assignment. 
\end{example}

\begin{example}[$n = 7$]
In this case, $\rho(\widetilde{\mathrm{M}}_{0,7})$ is two. There are three types of F-curves, $F_{4,1,1,1}$, $F_{3,2,1,1}$, and $F_{2,2,2,1}$. $F_{4,1,1,1}$ and $F_{2,2,2,1}$ generate two extremal rays of $\mathrm{NE}_{1}(\widetilde{\mathrm{M}}_{0,7})$. $\mathrm{Nef}(\widetilde{\mathrm{M}}_{0,7})$ is generated by $D_{2}+3D_{3}$ and $D_{2}+D_{3}$. The model corresponding to $D_{2}+3D_{3}$ is $\overline{\mathrm{M}}_{0, (1/3)^{7}}$. The model corresponding to $D_{2}+D_{3}$ is $\overline{\mathrm{M}}_{0,7}(Z)$ where $Z$ is the $S_{7}$-invariant extremal assignment contracting the spine of $F_{2,2,2,1}$. So all contractions of $\widetilde{\mathrm{M}}_{0,7}$ are obtained from extremal assignments. 
\end{example}

\begin{example}[$n=8$]
The nef cone is a 3-dimensional cone generated by 4 extremal rays $6D_{2}+11D_{3}+8D_{4}$, $3D_{2}+2D_{3}+4D_{4}$, $D_{2}+3D_{3}+6D_{4}$, and $2D_{2}+6D_{3}+5D_{4}$ (Figure \ref{fig:nefconefor8}). There are 8 proper faces of the nef cone. The corresponding birational contractions are listed in Table \ref{tbl:contraction8}. Therefore for $n \le 8$, the answer for Question \ref{que:modularity} is affirmative for $S_{n}$-invariant contractions. 

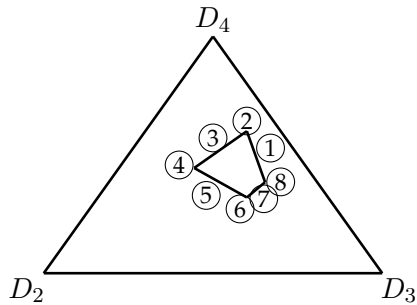
\begin{figure}[!ht]
\begin{tikzpicture}[scale=0.45]
\draw[line width = 1 pt] (0, 0) -- (10, 0);
\draw[line width = 1 pt] (0, 0) -- (5, 7);
\draw[line width = 1 pt] (5, 7) -- (10, 0);
\node (2) at (-0.5, -0.5) {$D_{2}$};
\node (3) at (10.5, -0.5) {$D_{3}$};
\node (4) at (5, 7.5) {$D_{4}$};
\draw[line width = 1 pt] (150/25, 56/25) -- (40/9, 28/9);
\draw[line width = 1 pt] (6,42/10) -- (40/9, 28/9);
\draw[line width = 1 pt] (6,42/10) -- (85/13,35/13);
\draw[line width = 1 pt] (150/25, 56/25) -- (85/13,35/13);
\footnotesize
\node (c1) at (6.7, 3.7) {$\circled{1}$};
\node (c8) at (6, 4.5) {$\circled{2}$};
\node (c7) at (5, 4) {$\circled{3}$};
\node (c6) at (4, 3.2) {$\circled{4}$};
\node (c5) at (4.8, 2.3) {$\circled{5}$};
\node (c4) at (5.8, 1.9) {$\circled{6}$};
\node (c3) at (6.5, 2.2) {$\circled{7}$};
\node (c2) at (7, 2.7) {$\circled{8}$};
\normalsize
\end{tikzpicture}
\caption{Nef cone and effective cone of $\widetilde{\mathrm{M}}_{0, 8}$}
\label{fig:nefconefor8}
\end{figure}
\begin{table}[!ht]
\begin{tabular}{|c|c|c|}\hline
face & contracted F-curves & birational model\\ \hline \hline
1&$F_{5,1,1,1}$ & $\overline{\mathrm{M}}_{0, ((1/3)^{8})}/S_{8}$ \\ \hline
2&$F_{4,2,1,1}, F_{5,1,1,1}$ & $(\PP^{1})^{8}\git \SL_{2}/S_{8} =\PP^{8}\git \SL_{2}$ \\ \hline
3&$F_{4,2,1,1}$ & Image of $\widetilde{\mathrm{M}}_{0, 8}$ in $((\PP^{1})^{8}\git \SL_{2} \times 
V_{0, (1/2)^{8}}^{3})/S_{8}$ \\ \hline
4&$F_{2,2,2,2}, F_{3,2,2,1}, F_{4,2,1,1}$ & $V_{0, (1/2)^{8}}^{3}/S_{8}$ \\ \hline
5&$F_{2,2,2,2}$ & $\overline{\mathrm{M}}_{0, 8}(Z)/S_{8}$ where $Z$ is generated by $\{2,2,2,2\}$ \\ \hline
6&$F_{3,3,1,1}, F_{2,2,2,2}$ & $V_{0, (3/8)^{8}}^{2}/S_{8}$ \\ \hline
7&$F_{3,3,1,1}$ & $\overline{\mathrm{M}}_{0, 8}(Z)/S_{8}$ where $Z$ is generated by $\{3,3,1,1\}$ \\ \hline
8&$F_{5,1,1,1}, F_{3,3,1,1}$ & $V_{1/3, (1/3)^{8}}^{2}/S_{8}$ \\ \hline
\end{tabular}
\medskip
\caption{List of contractions of $\widetilde{\mathrm{M}}_{0, 8}$}
\label{tbl:contraction8}
\end{table}
\end{example}

\begin{example}[$n=9$]
The nef cone is a 3-dimensional cone generated by 4 extremal rays $3D_{2}+3D_{3}+4D_{4}$, $D_{2}+D_{3}+2D_{4}$, $D_{2}+3D_{3}+6D_{4}$, and $D_{2}+3D_{3}+2D_{4}$ (Figure \ref{fig:nefconefor9}). There are 8 faces of the nef cone. The corresponding birational contractions are listed in Table \ref{tbl:contraction9}.

Note that three faces are obtained from neither extremal assignments nor Veronese quotients. By using results in Section \ref{sec:invariant}, it is straightforward to see that there is no $S_{n}$-invariant extremal assignment which contracts the given F-curves only. On the other hand, Veronese quotients may provide contractions which are not from extremal assignments. This occurs if there are strictly semistable points, or equivalently, if the quantity $\frac{c_{T}-1}{1-\gamma}$ has an integer value for some $T$ (Remark \ref{rem:semistable}). If it is an integer $i$ and $|T| = k$, then 
\[
	(d-1)\gamma + nc = d+1, \quad \frac{kc-1}{1-\gamma} = i
\]
where $\vec{c} = (c, c, \cdots, c)$. We can find all possible solutions for $\gamma$ and $c$ if we fix $d$, $k$, and $i$. Furthermore, if $d > 2n - 3$, then any parametrized curve has nodal singularities without marked points only (\cite[Corollary 2.7]{GJM13}). Therefore the possible models are Hassett's moduli spaces of weighted stable curves. Thus it suffices to check finitely many $d$, $k$, and $i$. For each solution, by using \cite[Theorem 2.1]{GJMS13}, one can compute the set of contracted F-curves. 

Therefore we can conclude that the answer for Question \ref{que:modularity} is negative for $n \ge 9$ in general.

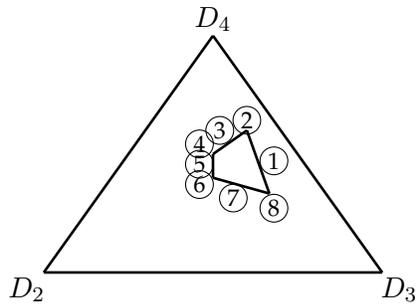
\begin{figure}[!ht]
\begin{tikzpicture}[scale=0.45]
\draw[line width = 1 pt] (0, 0) -- (10, 0);
\draw[line width = 1 pt] (0, 0) -- (5, 7);
\draw[line width = 1 pt] (5, 7) -- (10, 0);
\node (2) at (-0.5, -0.5) {$D_{2}$};
\node (3) at (10.5, -0.5) {$D_{3}$};
\node (4) at (5, 7.5) {$D_{4}$};
\draw[line width = 1 pt] (5, 28/10) -- (5, 14/4);
\draw[line width = 1 pt] (5, 14/4) -- (6, 42/10);
\draw[line width = 1 pt] (6, 42/10) -- (40/6, 14/6);
\draw[line width = 1 pt] (40/6, 14/6) -- (5, 28/10);
\footnotesize
\node (c1) at (6.8, 3.3) {$\circled{1}$};
\node (c2) at (6, 4.5) {$\circled{2}$};
\node (c3) at (5.2, 4.2) {$\circled{3}$};
\node (c4) at (4.6, 3.8) {$\circled{4}$};
\node (c5) at (4.6, 3.2) {$\circled{5}$};
\node (c6) at (4.6, 2.6) {$\circled{6}$};
\node (c7) at (5.6, 2.2) {$\circled{7}$};
\node (c8) at (6.8, 1.9) {$\circled{8}$};
\normalsize
\end{tikzpicture}
\caption{Nef cone and effective cone of $\widetilde{\mathrm{M}}_{0, 9}$}
\label{fig:nefconefor9}
\end{figure}
\begin{table}[!ht]
\begin{tabular}{|c|c|c|}\hline
face & contracted F-curves & birational model\\ \hline \hline
1&$F_{6,1,1,1}$ & $\overline{\mathrm{M}}_{0, ((1/3)^{9})}/S_{9}$ \\ \hline
2&$F_{6,1,1,1}, F_{5,2,1,1}$ & $\overline{\mathrm{M}}_{0, ((1/4)^{9})}/S_{9} = (\PP^{1})^{9}\git \SL_{2}/S_{9} = \PP^{9}\git \SL_{2}$ \\ \hline
3&$F_{5,2,1,1}$ & not obtained from assignments or GIT \\ \hline
4&$F_{5,2,1,1}, F_{4,2,2,1}$ & not obtained from assignments or GIT \\ \hline
5&$F_{4,2,2,1}$ & not obtained from assignments or GIT \\ \hline
6&$F_{4,2,2,1}, F_{3,2,2,2}$ & $\overline{\mathrm{M}}_{0, 9}(Z)/S_{9}$ where $Z$ is generated by $\{2, 2, 2, 2, 1\}$ \\ \hline
7&$F_{3,2,2,2}$ & $\overline{\mathrm{M}}_{0, 9}(Z)/S_{9}$ where $Z$ is generated by $\{3,2,2,2\}$ \\ \hline
8&$F_{3,2,2,2}, F_{6,1,1,1}, F_{4,3,1,1}, F_{3,3,2,1}$ & $V_{0, (1/3)^{9}}^{2}/S_{9}$ \\ \hline
\end{tabular}
\medskip
\caption{List of contractions of $\widetilde{\mathrm{M}}_{0, 9}$}
\label{tbl:contraction9}
\end{table}
\end{example}

\begin{example}[$n=10$]
In this case, the Picard number is 4. The $f$-vector of a slice of the nef cone is $(1, 7, 13, 8, 1)$ (Figure \ref{fig:slice}). In particular, the nef cone has 7 extremal rays. Table \ref{tbl:contraction10} gives us the list of contracted curves and the corresponding birational models for all extremal rays.

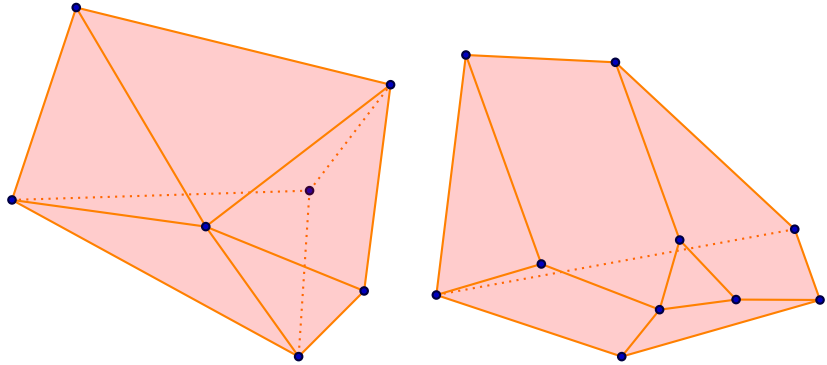
\begin{figure}[!ht]
\begin{tikzpicture}%
	[x={(0.393085cm, -0.360638cm)},
	y={(0.899030cm, 0.343769cm)},
	z={(-0.192951cm, 0.867043cm)},
	scale=30.000000,
	back/.style={dotted, thick},
	edge/.style={color=orange, thick},
	facet/.style={fill=red,fill opacity=0.200000},
	vertex/.style={inner
sep=1pt,circle,draw=blue!25!black,fill=blue!75!black,thick,anchor=base}]
%
%
\coordinate (0.0909, 0.273, 0.273) at (0.0909, 0.273, 0.273);
\coordinate (0.0500, 0.150, 0.300) at (0.0500, 0.150, 0.300);
\coordinate (0.154, 0.231, 0.231) at (0.154, 0.231, 0.231);
\coordinate (0.105, 0.316, 0.316) at (0.105, 0.316, 0.316);
\coordinate (0.0645, 0.194, 0.387) at (0.0645, 0.194, 0.387);
\coordinate (0.235, 0.176, 0.353) at (0.235, 0.176, 0.353);
\coordinate (0.174, 0.261, 0.261) at (0.174, 0.261, 0.261);
\draw[edge,back] (0.0909, 0.273, 0.273) -- (0.0500, 0.150, 0.300);
\draw[edge,back] (0.0909, 0.273, 0.273) -- (0.154, 0.231, 0.231);
\draw[edge,back] (0.0909, 0.273, 0.273) -- (0.105, 0.316, 0.316);
\node[vertex] at (0.0909, 0.273, 0.273)     {};
\fill[facet] (0.235, 0.176, 0.353) -- (0.105, 0.316, 0.316) -- (0.0645,
0.194, 0.387) -- cycle {};
\fill[facet] (0.235, 0.176, 0.353) -- (0.0500, 0.150, 0.300) -- (0.154,
0.231, 0.231) -- cycle {};
\fill[facet] (0.235, 0.176, 0.353) -- (0.0500, 0.150, 0.300) -- (0.0645,
0.194, 0.387) -- cycle {};
\fill[facet] (0.174, 0.261, 0.261) -- (0.154, 0.231, 0.231) -- (0.235,
0.176, 0.353) -- cycle {};
\fill[facet] (0.174, 0.261, 0.261) -- (0.105, 0.316, 0.316) -- (0.235,
0.176, 0.353) -- cycle {};
\draw[edge] (0.0500, 0.150, 0.300) -- (0.154, 0.231, 0.231);
\draw[edge] (0.0500, 0.150, 0.300) -- (0.0645, 0.194, 0.387);
\draw[edge] (0.0500, 0.150, 0.300) -- (0.235, 0.176, 0.353);
\draw[edge] (0.154, 0.231, 0.231) -- (0.235, 0.176, 0.353);
\draw[edge] (0.154, 0.231, 0.231) -- (0.174, 0.261, 0.261);
\draw[edge] (0.105, 0.316, 0.316) -- (0.0645, 0.194, 0.387);
\draw[edge] (0.105, 0.316, 0.316) -- (0.235, 0.176, 0.353);
\draw[edge] (0.105, 0.316, 0.316) -- (0.174, 0.261, 0.261);
\draw[edge] (0.0645, 0.194, 0.387) -- (0.235, 0.176, 0.353);
\draw[edge] (0.235, 0.176, 0.353) -- (0.174, 0.261, 0.261);
\node[vertex] at (0.0500, 0.150, 0.300)     {};
\node[vertex] at (0.154, 0.231, 0.231)     {};
\node[vertex] at (0.105, 0.316, 0.316)     {};
\node[vertex] at (0.0645, 0.194, 0.387)     {};
\node[vertex] at (0.235, 0.176, 0.353)     {};
\node[vertex] at (0.174, 0.261, 0.261)     {};
\end{tikzpicture}
\quad
\begin{tikzpicture}%
	[x={(0.374551cm, -0.193627cm)},
	y={(0.854879cm, 0.450783cm)},
	z={(-0.359018cm, 0.871380cm)},
	scale=30.000000,
	back/.style={dotted, thick},
	edge/.style={color=orange, thick},
	facet/.style={fill=red,fill opacity=0.200000},
	vertex/.style={inner
sep=1pt,circle,draw=blue!25!black,fill=blue!75!black,thick,anchor=base}]
%
%
\coordinate (0.167, 0.167, 0.333) at (0.167, 0.167, 0.333);
\coordinate (0.0667, 0.200, 0.400) at (0.0667, 0.200, 0.400);
\coordinate (0.0500, 0.150, 0.300) at (0.0500, 0.150, 0.300);
\coordinate (0.200, 0.200, 0.300) at (0.200, 0.200, 0.300);
\coordinate (0.100, 0.300, 0.267) at (0.100, 0.300, 0.267);
\coordinate (0.0857, 0.257, 0.371) at (0.0857, 0.257, 0.371);
\coordinate (0.114, 0.200, 0.257) at (0.114, 0.200, 0.257);
\coordinate (0.171, 0.229, 0.314) at (0.171, 0.229, 0.314);
\coordinate (0.133, 0.289, 0.244) at (0.133, 0.289, 0.244);
\coordinate (0.180, 0.240, 0.280) at (0.180, 0.240, 0.280);
\draw[edge,back] (0.0500, 0.150, 0.300) -- (0.100, 0.300, 0.267);
\fill[facet] (0.180, 0.240, 0.280) -- (0.171, 0.229, 0.314) -- (0.0857,
0.257, 0.371) -- (0.100, 0.300, 0.267) -- (0.133, 0.289, 0.244) -- cycle
{};
\fill[facet] (0.180, 0.240, 0.280) -- (0.200, 0.200, 0.300) -- (0.114,
0.200, 0.257) -- (0.133, 0.289, 0.244) -- cycle {};
\fill[facet] (0.114, 0.200, 0.257) -- (0.0500, 0.150, 0.300) -- (0.167,
0.167, 0.333) -- (0.200, 0.200, 0.300) -- cycle {};
\fill[facet] (0.171, 0.229, 0.314) -- (0.200, 0.200, 0.300) -- (0.167,
0.167, 0.333) -- (0.0667, 0.200, 0.400) -- (0.0857, 0.257, 0.371) --
cycle {};
\fill[facet] (0.0500, 0.150, 0.300) -- (0.167, 0.167, 0.333) -- (0.0667,
0.200, 0.400) -- cycle {};
\fill[facet] (0.180, 0.240, 0.280) -- (0.200, 0.200, 0.300) -- (0.171,
0.229, 0.314) -- cycle {};
\draw[edge] (0.167, 0.167, 0.333) -- (0.0667, 0.200, 0.400);
\draw[edge] (0.167, 0.167, 0.333) -- (0.0500, 0.150, 0.300);
\draw[edge] (0.167, 0.167, 0.333) -- (0.200, 0.200, 0.300);
\draw[edge] (0.0667, 0.200, 0.400) -- (0.0500, 0.150, 0.300);
\draw[edge] (0.0667, 0.200, 0.400) -- (0.0857, 0.257, 0.371);
\draw[edge] (0.0500, 0.150, 0.300) -- (0.114, 0.200, 0.257);
\draw[edge] (0.200, 0.200, 0.300) -- (0.114, 0.200, 0.257);
\draw[edge] (0.200, 0.200, 0.300) -- (0.171, 0.229, 0.314);
\draw[edge] (0.200, 0.200, 0.300) -- (0.180, 0.240, 0.280);
\draw[edge] (0.100, 0.300, 0.267) -- (0.0857, 0.257, 0.371);
\draw[edge] (0.100, 0.300, 0.267) -- (0.133, 0.289, 0.244);
\draw[edge] (0.0857, 0.257, 0.371) -- (0.171, 0.229, 0.314);
\draw[edge] (0.114, 0.200, 0.257) -- (0.133, 0.289, 0.244);
\draw[edge] (0.171, 0.229, 0.314) -- (0.180, 0.240, 0.280);
\draw[edge] (0.133, 0.289, 0.244) -- (0.180, 0.240, 0.280);
\node[vertex] at (0.167, 0.167, 0.333)     {};
\node[vertex] at (0.0667, 0.200, 0.400)     {};
\node[vertex] at (0.0500, 0.150, 0.300)     {};
\node[vertex] at (0.200, 0.200, 0.300)     {};
\node[vertex] at (0.100, 0.300, 0.267)     {};
\node[vertex] at (0.0857, 0.257, 0.371)     {};
\node[vertex] at (0.114, 0.200, 0.257)     {};
\node[vertex] at (0.171, 0.229, 0.314)     {};
\node[vertex] at (0.133, 0.289, 0.244)     {};
\node[vertex] at (0.180, 0.240, 0.280)     {};
\end{tikzpicture}
\caption{Slices of the nef cones of $\widetilde{\mathrm{M}}_{0, 10}$ and $\widetilde{\mathrm{M}}_{0, 11}$}
\label{fig:slice}
\end{figure}

\begin{table}[!ht]
\begin{tabular}{|c|c|}\hline
contracted F-curves & birational model\\ \hline \hline
$F_{4,2,2,2}, F_{3,3,3,1}, F_{3,3,2,2}$ & not obtained from assignments or GIT \\ \hline
$F_{7,1,1,1}, F_{4,4,1,1}, F_{3,3,3,1}, F_{3,3,2,2}$ & $V_{0, (3/10)^{10}}^{2}/S_{10}$\\ \hline
$F_{5,3,1,1}, F_{5,2,2,1}, F_{4,2,2,2}, F_{3,3,3,1}$ & $V_{0, (2/5)^{10}}^{3}/S_{10}$ \\ \hline
$F_{7,1,1,1}, F_{5,3,1,1}, F_{3,3,3,1}$ & $V_{1/3, (1/3)^{10}}^{3}/S_{10}$ \\ \hline
$F_{7,1,1,1}, F_{6,2,1,1}, F_{5,3,1,1}, F_{5,2,2,1}$ & $(\PP^{1})^{10}\git \SL_{2}/S_{10} = \PP^{10}\git \SL_{2}$ \\ \hline
$F_{6,2,1,1}, F_{5,2,2,1}, F_{4,4,1,1}, F_{4,3,2,1}, F_{4,2,2,2}, F_{3,3,2,2}$ & $V_{0, (1/2)^{10}}^{4}/S_{10}$ \\ \hline
$F_{7,1,1,1}, F_{6,2,1,1}, F_{4,4,1,1}$ & $V_{1/2, (1/4)^{10}}^{2}/S_{10}$ \\ \hline
\end{tabular}
\medskip
\caption{List of contractions of $\widetilde{\mathrm{M}}_{0, 10}$ corresponding to extremal rays}
\label{tbl:contraction10}
\end{table}
\end{example}

\begin{example}[$n=11$ case]
Again, the Picard number is 4 and the $f$-vector of a slice of the nef cone is $(1, 10, 16, 8, 1)$ (Figure \ref{fig:slice}). So the nef cone has 10 extremal rays. Table \ref{tbl:contraction11} is the list of contracted curves and the corresponding birational models for all extremal rays.

\begin{table}[!ht]
\begin{tabular}{|c|c|}\hline
contracted F-curves & birational model\\ \hline \hline
$F_{4,4,2,1}, F_{4,3,2,2}, F_{3,3,3,2}$ & not obtained from assignments or GIT \\ \hline
$F_{8,1,1,1}, F_{5,4,1,1}, F_{4,4,2,1}, F_{3,3,3,2}$ & $V_{0, (3/11)^{11}}^{2}/S_{11}$\\ \hline
$F_{5,2,2,2}, F_{4,3,2,2}, F_{3,3,3,2}$ & $\overline{\mathrm{M}}_{0, 11}(Z)/S_{11}$ for $Z$ generated by $\{3, 2, 2, 2\}$ and $\{3, 3, 3, 2\}$\\ \hline
$F_{6,3,1,1}, F_{5,2,2,2}, F_{4,3,3,1}, F_{3,3,3,2}$ & $V_{0, (4/11)^{11}}^{3}/S_{11}$\\ \hline
$F_{8,1,1,1}, F_{6,3,1,1}, F_{4,3,3,1}, F_{3,3,3,2}$ & $V_{1/6, (1/3)^{11}}^{3}/S_{11}$ \\ \hline
$F_{6,2,2,1}, F_{5,2,2,2}, F_{4,4,2,1}, F_{4,3,2,2}$ & $V_{0, (5/11)^{11}}^{4}/S_{11}$ \\ \hline
$F_{7,2,1,1}, F_{6,2,2,1}, F_{4,4,2,1}$ & not obtained from assignments or GIT \\ \hline
$F_{8,1,1,1}, F_{7,2,1,1}, F_{5,4,1,1}, F_{4,4,2,1}$ & $V_{1/4, (1/4)^{11}}^{2}/S_{11}$ \\ \hline
$F_{6,3,1,1}, F_{6,2,2,1}, F_{5,2,2,2}$ & not obtained from assignments or GIT \\ \hline
$F_{8,1,1,1}, F_{7,2,1,1}, F_{6,3,1,1}, F_{6,2,2,1}$ & $\overline{\mathrm{M}}_{0, ((1/5)^{11})}/S_{11} = (\PP^{1})^{11}\git \SL_{2}/S_{11} = \PP^{11}\git \SL_{2}$ \\ \hline
\end{tabular}
\medskip
\caption{List of contractions of $\widetilde{\mathrm{M}}_{0, 11}$ corresponding to extremal rays}
\label{tbl:contraction11}
\end{table}
\end{example}


\section{Non-projective examples}\label{sec:nonprojectiveexamples}

A priori, for an extremal assignment $Z$, the birational contraction $\Mznb(Z)$ exists in the category of algebraic spaces (Theorem \ref{thm:extassignmentmodel}). There are indeed many examples in which $\Mznb(Z)$ is not a projective variety. In this section, we present several examples of non-projective birational models. 

There is a smooth extremal assignment $Z$ which provides a non-projective birational model. 

\begin{example}[Non-projective smooth assignment]\label{ex:nonprojective1}
Let $Z$ be a smooth extremal assignment of order 6 whose maximal elements in the contraction indicator are
\[
	\{1, 2, 3, 4\}, \{1, 2, 5\}, \{3, 4, 5\}, 	\{2, 3, 6\}, \{1, 4, 6\}.
\]
Then it is straightforward to check that the contracted F-curves by $\pi_{Z} : \overline{\mathrm{M}}_{0, 6} \to \overline{\mathrm{M}}_{0, 6}(Z)$ are
\[
	F_{\{1\}, \{2\}, \{3\}, \{4,5,6\}},
	F_{\{1\}, \{2\}, \{4\}, \{3,5,6\}},
	F_{\{1\}, \{2\}, \{5\}, \{3,4,6\}}, 
	F_{\{1\}, \{3\}, \{4\}, \{2,5,6\}},
	F_{\{1\}, \{4\}, \{6\}, \{2,3,5\}},
\]
\[
	F_{\{2\}, \{3\}, \{4\}, \{1,5,6\}},
	F_{\{2\}, \{3\}, \{6\}, \{1,4,5\}},
	F_{\{3\}, \{4\}, \{5\}, \{1,2,6\}},
	F_{\{1\}, \{2\}, \{3,4\}, \{5,6\}},
	F_{\{1\}, \{3\}, \{2,4\}, \{5,6\}},
\]
\[
	F_{\{1\}, \{4\}, \{2,3\}, \{5,6\}},
	F_{\{2\}, \{3\}, \{1,4\}, \{5,6\}},
	F_{\{2\}, \{4\}, \{1,3\}, \{5,6\}},
	F_{\{3\}, \{4\}, \{1,2\}, \{5,6\}}.
\]
Note that the curve cone $\mathrm{NE}_{1}(\overline{\mathrm{M}}_{0, 6})$ is generated by F-curves (\cite[Theorem 1.2]{KM13}). If $\overline{\mathrm{M}}_{0, 6}(Z)$ is projective, then there is $N \in \mathrm{Nef}(\overline{\mathrm{M}}_{0, 6})$ such that for an F-curve $F$, $N \cdot F = 0$ if and only if $F$ is in the above list. By using a computer program, it can be shown that there is no such divisor $N$.
\end{example}

\begin{example} 
By a similar idea, one can check that the smooth extremal assignment $Z$ in Example \ref{ex:noweightassignment} gives a non-projective birational contraction $\overline{\mathrm{M}}_{0, 7}(Z)$. 
\end{example}

\begin{example}[Non-projective $S_{n}$-invariant assignment]
Let $Z$ be an $S_{12}$-invariant extremal assignment of order 12, which corresponds to integer partitions $\{7,3,1,1\}$ and $\{3,3,3,3\}$. Then the contracted F-curves are exactly $F_{7,3,1,1}$ and $F_{3,3,3,3}$. The $S_{n}$-invariant F-conjecture is true up to $n \le 24$ (\cite[Theorem 6.1]{Gib09}), so $\mathrm{NE}_{1}(\widetilde{\mathrm{M}}_{0, 12})$ is generated by F-curves. By using a similar computation with Example \ref{ex:nonprojective1}, we can check that $\overline{\mathrm{M}}_{0, 12}(Z)$ is not projective. 
\end{example}

\begin{remark}
However, by Proposition \ref{prop:invariantdivisorial}, every $S_{n}$-invariant smooth extremal assignment gives a projective birational contraction. 
\end{remark}

We finish this section with a partial positive result showing that some birational models are indeed varieties. 

\begin{proposition}\label{prop:scheme}
Let $Z$ be a smooth extremal assignment. Then $\Mznb(Z)$ is a proper variety. 
\end{proposition}

\begin{proof}
Since we already know that $\Mznb(Z)$ is a smooth proper algebraic space, it suffices to show that it is a scheme. Let $X := (C, x_{1}, x_{2}, \cdots, x_{n}) \in \Mznb(Z)$. Since $Z$ is smooth, $C$ is obtained from $(C^{s}, x_{1}^{s}, x_{2}^{s}, \cdots, x_{n}^{s}) \in \Mznb$ by contracting some (not necessarily irreducible) tails. Thus if we define weight data $A$ such that the sum over each tail is at most one, then $X \in \Mza$. Furthermore, since a deformation of $X$ is obtained by perturbing marked points and resolving some singularities, any small deformation of $X$ is also in $\Mza$. Therefore there is a Zariski open neighborhood $U$ of $X$ in $\Mznb(Z)$, which is isomorphic to an open neighborhood of $X$ in $\Mza$. Since $\Mza$ is a projective variety, $U$ is a scheme. Therefore any point of $\Mznb(Z)$ has a Zariski open neighborhood which is a scheme. This implies $\Mznb(Z)$ is a scheme.
\end{proof}

\begin{remark}
Therefore, by using extremal assignments, we obtain many examples of smooth proper and non-projective varieties. 
\end{remark}

We expect that this is true for any extremal assignment $Z$. 

\begin{conjecture}
For any extremal assignment $Z$, $\Mznb(Z)$ is a proper variety.
\end{conjecture}


\bibliographystyle{alpha}
\newcommand{\etalchar}[1]{$^{#1}$}


\begin{thebibliography}{GJMS13}

\bibitem[Ber89]{Ber89}
Claude Berge.
\newblock {\em Hypergraphs}, volume~45 of {\em North-Holland Mathematical
  Library}.
\newblock North-Holland Publishing Co., Amsterdam, 1989.
\newblock Combinatorics of finite sets, Translated from the French.

\bibitem[BH11]{BH11}
Gilberto Bini and John Harer.
\newblock Euler characteristics of moduli spaces of curves.
\newblock {\em J. Eur. Math. Soc. (JEMS)}, 13(2):487--512, 2011.

\bibitem[BM13]{BM13}
Jonas Bergstr{{\"o}}m and Satoshi Minabe.
\newblock On the cohomology of moduli spaces of (weighted) stable rational
  curves.
\newblock {\em Math. Z.}, 275(3-4):1095--1108, 2013.

\bibitem[BM14]{BM14}
Jonas Bergstr{{\"o}}m and Satoshi Minabe.
\newblock On the cohomology of the {L}osev-{M}anin moduli space.
\newblock {\em Manuscripta Math.}, 144(1-2):241--252, 2014.

\bibitem[Bog99]{Bog99}
Marco Boggi.
\newblock Compactifications of configurations of points on {${\Bbb P}^1$} and
  quadratic transformations of projective space.
\newblock {\em Indag. Math. (N.S.)}, 10(2):191--202, 1999.

\bibitem[Cey09]{Cey09}
{{\"O}}zg{{\"u}}r Ceyhan.
\newblock Chow groups of the moduli spaces of weighted pointed stable curves of
  genus zero.
\newblock {\em Adv. Math.}, 221(6):1964--1978, 2009.

\bibitem[CT15]{CT15}
Ana-Maria Castravet and Jenia Tevelev.
\newblock {$\overline{M}_{0,n}$} is not a {M}ori dream space.
\newblock {\em Duke Math. J.}, 164(8):1641--1667, 2015.

\bibitem[dRJR14]{dRJR14}
Rodrigo~Ferreira da~Rosa, David Jensen, and Dhruv Ranganathan.
\newblock Toric graph associahedra and compactifications of $\rm M_{0,n}$.
\newblock arXiv:1411.0537, 2014.

\bibitem[FN72]{FN71}
Akira Fujiki and Shigeo Nakano.
\newblock Supplement to ``{O}n the inverse of monoidal transformation''.
\newblock {\em Publ. Res. Inst. Math. Sci.}, 7:637--644, 1971/72.

\bibitem[Ful93]{Ful93}
William Fulton.
\newblock {\em Introduction to toric varieties}, volume 131 of {\em Annals of
  Mathematics Studies}.
\newblock Princeton University Press, Princeton, NJ, 1993.
\newblock The William H. Roever Lectures in Geometry.

\bibitem[Gib09]{Gib09}
Angela Gibney.
\newblock Numerical criteria for divisors on {$\overline M_g$} to be ample.
\newblock {\em Compos. Math.}, 145(5):1227--1248, 2009.

\bibitem[GJM13]{GJM13}
Noah Giansiracusa, David Jensen, and Han-Bom Moon.
\newblock G{IT} compactifications of {$\bar{M}_{0,n}$} and flips.
\newblock {\em Adv. Math.}, 248:242--278, 2013.

\bibitem[GJMS13]{GJMS13}
Angela Gibney, David Jensen, Han-Bom Moon, and David Swinarski.
\newblock Veronese quotient models of {$\overline{\rm M}_{0,n}$} and conformal
  blocks.
\newblock {\em Michigan Math. J.}, 62(4):721--751, 2013.

\bibitem[Has03]{Has03}
Brendan Hassett.
\newblock Moduli spaces of weighted pointed stable curves.
\newblock {\em Adv. Math.}, 173(2):316--352, 2003.

\bibitem[Kap93a]{Kap93b}
M.~M. Kapranov.
\newblock Chow quotients of {G}rassmannians. {I}.
\newblock In {\em I. {M}. {G}elfand {S}eminar}, volume~16 of {\em Adv. Soviet
  Math.}, pages 29--110. Amer. Math. Soc., Providence, RI, 1993.

\bibitem[Kap93b]{Kap93a}
M.~M. Kapranov.
\newblock Veronese curves and {G}rothendieck-{K}nudsen moduli space {$\overline
  M_{0,n}$}.
\newblock {\em J. Algebraic Geom.}, 2(2):239--262, 1993.

\bibitem[KM13]{KM13}
Se{{\'a}}n Keel and James McKernan.
\newblock Contractible extremal rays on {$\overline M_{0,n}$}.
\newblock In {\em Handbook of moduli. {V}ol. {II}}, volume~25 of {\em Adv.
  Lect. Math. (ALM)}, pages 115--130. Int. Press, Somerville, MA, 2013.

\bibitem[LM00]{LM00}
A.~Losev and Y.~Manin.
\newblock New moduli spaces of pointed curves and pencils of flat connections.
\newblock {\em Michigan Math. J.}, 48:443--472, 2000.
\newblock Dedicated to William Fulton on the occasion of his 60th birthday.

\bibitem[Man95]{Man95}
Yu.~I. Manin.
\newblock Generating functions in algebraic geometry and sums over trees.
\newblock In {\em The moduli space of curves ({T}exel {I}sland, 1994)}, volume
  129 of {\em Progr. Math.}, pages 401--417. Birkh{\"a}user Boston, Boston, MA,
  1995.

\bibitem[Moo11]{Moo11}
Han-Bom Moon.
\newblock {\em Birational geometry of moduli spaces of curves of genus zero}.
\newblock PhD thesis, Seoul National University, 2011.

\bibitem[Moo15]{Moo15b}
Han-Bom Moon.
\newblock Mori's program for {$\overline {\rm M}_{0,6}$} with symmetric
  divisors.
\newblock {\em Math. Nachr.}, 288(7):824--836, 2015.

\bibitem[S{\etalchar{+}}]{Sage}
W.~A. Stein et~al.
\newblock {\em Sage}.
\newblock The Sage Development Team, {S}age {M}athematics {S}oftware ({V}ersion
  6.5).

\bibitem[Sek96]{Sek96}
J.~Sekiguchi.
\newblock Cross ratio varieties for root systems of type {$A$} and the {T}erada
  model.
\newblock {\em J. Math. Sci. Univ. Tokyo}, 3(1):181--197, 1996.

\bibitem[Smy13]{Smy13}
David~Ishii Smyth.
\newblock Towards a classification of modular compactifications of {$M_{g,n}$}.
\newblock {\em Invent. Math.}, 192(2):459--503, 2013.

\end{thebibliography}

\end{document}